\documentclass[10pt]{amsart}

\usepackage{color,epsf}
\usepackage[dvipsnames]{xcolor}
\usepackage{nicefrac}
\usepackage{dsfont}
\usepackage{enumerate}
\usepackage{tikz}
\usepackage{pgfplots}
\usepackage{comment}
\usepackage{float}
\usepackage{mathtools}
\usepackage{graphicx}
\usepackage{bm}
\usepackage{amsmath}
\usepackage{tabularx}
\usepackage{lmodern}
\usepackage{color, colortbl}
\usepackage{url}

\mathtoolsset{showonlyrefs}

\usepackage{graphicx}
   \usepackage[bookmarksopen=false,pdftex=true,breaklinks=true,%
      backref=page,pagebackref=true,plainpages=false,%
      hyperindex=true,pdfstartview=FitH,colorlinks=true,%
      pdfpagelabels=true,colorlinks=true,linkcolor=blue,%
      citecolor=red,urlcolor=green,hypertexnames=false%
      ]%
    {hyperref}

\includecomment{illustration}
\newcommand{\be}{\begin{equation}}
\newcommand{\ee}{\end{equation}}

\usepackage{amsmath,amssymb,amsfonts}

\newtheorem{theorem}{Theorem}
\newtheorem{lemma}[theorem]{Lemma}
\newtheorem{corollary}[theorem]{Corollary}
\newtheorem{proposition}[theorem]{Proposition}

\theoremstyle{definition}
\newtheorem{definition}{Definition}

\theoremstyle{remark}
\newtheorem{remark}{Remark}

\definecolor{DarkBlue}{rgb}{0,0.1,0.55}

\numberwithin{equation}{section}

\newcommand {\hide}[1]{}

\newcommand {\junk}[1]{}

\newcommand {\R} {\mathbb{R}}

\newcommand {\Z}  {\mathbb{Z}}

\newcommand {\eps} {{\varepsilon}}

\newcommand {\PP}     {\mathbb{P}} %projective space

\def\addots{\mathinner{\mkern1mu
\raise1pt\vbox{\kern7pt\hbox{.}}
\mkern2mu\raise4pt\hbox{.}\mkern2mu
\raise7pt\hbox{.}\mkern1mu}}

\newcommand{\HH}  {\mbox{\rm H}}

\newcommand {\RP}   {\mathbb{R}\mathrm{P}}

\newcommand{\beq}{\begin{equation}}
\newcommand{\eeq}{\begin{equation}}

\newcommand{\Ker}[1]{\mathrm{Ker }\left(#1\right)}
\newcommand{\Img}[1]{\mathrm{Im }\left(#1\right)}

\newcommand*{\EX}[1]{\mathbb{E}\left[#1\right]}
\newcommand*{\PR}[1]{\mathbb{P}\left[#1\right]}
\newcommand*{\condition}{\,\middle|\,}

\newcommand{\indicator}[1]{\ensuremath{\mathds{1}\left\{#1\right\}}}

\newcommand{\bigslant}[2]{{\raisebox{.2em}{$#1$}\left/\raisebox{-.2em}{$#2$}\right.}}
\newcommand*{\de}[1]{\,\mathrm{d}#1}
\newcommand{\eqcomment}[1]{\hfill &&\text{(#1)}}
\newcommand{\Pos}{\mathcal{P}}
\newcommand{\graph}{\mathcal{G}}
\newcommand{\QQ}{\mathcal{Q}}

\newcommand{\vol}[1]{\mathrm{vol}\left(#1\right)}

\newif\ifnotes\notestrue %set this to true if notes are visible and to false (next line) if they should be hidden
\ifnotes

\fi

\newcommand{\spherecolor}{CornflowerBlue!70!white}
\newcommand{\sphereopacity}{1}
\newcommand{\edgecolor}{green}
\newcommand{\nonedgecolor}{red}
\newcommand{\pncolor}{yellow}
\newcommand{\pnopacity}{1.0}
\newcommand{\pntildeopacity}{0.4}
\newcommand{\smallergoodconecolor}{Brown}
\newcommand{\goodconecolor}{green!30!white}
\newcommand{\goodconeopacity}{0.5}

\usetikzlibrary{shapes.geometric,pgfplots.polar,intersections,math,arrows,snakes,arrows.meta}
\tikzset{
    invclip/.style={
        clip,
        insert path={{[reset cm](-16383.99999pt,-16383.99999pt) rectangle (16383.99999pt,16383.99999pt)}}
    }
}
\pgfplotsset{
  compat=newest,
  xlabel near ticks,
  ylabel near ticks
}
\usepgfplotslibrary{patchplots,polar,colormaps}

\newcommand{\geodesicA}[6]{
          \begin{scope}[rotate=#5]
            \path[#6,evaluate={
										\startangle = (#1);	\endangle = (#2);
										\xradius = (#3);		\yradius = (#4);
										\startx = \xradius * cos(\startangle);
										\starty = \yradius * sin(\startangle);
									},] (\startx,\starty) arc [	start angle=\startangle,
																		end angle=\endangle,
															            x radius=\xradius,
															            y radius=\yradius] ;
          \end{scope}
}

\newcommand{\vertex}[4]{
          \begin{scope}[rotate=#4]
		    \fill[evaluate={
										\xxx = (#1) * cos(#3);
										\yyy = (#2) * sin(#3);
									},] (\xxx,\yyy) circle (1pt);
          \end{scope}
}

\newcommand{\nonedgeA}[5]{\geodesicA{#1}{#2}{#3}{#4}{#5}{draw,thick,dotted,\nonedgecolor}} 

\newcommand{\edgeA}[5]{\geodesicA{#1}{#2}{#3}{#4}{#5}{draw,thick,\edgecolor}}

\pgfkeys{/tikz/.cd,
    contour distance/.store in=\ContourDistance,
    contour distance=-10pt, % for the other orientation use a +
    contour step/.store in=\ContourStep,
    contour step=1pt,
}

\pgfdeclaredecoration{closed contour}{initial}
{% 
\state{initial}[width=\ContourStep,next state=cont] {
    \pgfmoveto{\pgfpoint{\ContourStep}{\ContourDistance}}
    \pgfcoordinate{first}{\pgfpoint{\ContourStep}{\ContourDistance}}
    \pgfpathlineto{\pgfpoint{0.3\pgflinewidth}{\ContourDistance}}
    \pgfcoordinate{lastup}{\pgfpoint{1pt}{\ContourDistance}}
    
  }
  \state{cont}[width=\ContourStep]{
     \pgfmoveto{\pgfpointanchor{lastup}{center}}
     \pgfpathlineto{\pgfpoint{\ContourStep}{\ContourDistance}}
     \pgfcoordinate{lastup}{\pgfpoint{\ContourStep}{\ContourDistance}}
  }
  \state{final}[width=\ContourStep]
  { % perhaps unnecessary but doesn't hurt either
    \pgfmoveto{\pgfpointanchor{lastup}{center}}
    \pgfpathlineto{\pgfpointanchor{first}{center}}
  }
}

\makeatletter
\newcommand{\@giventhatstar}[2]{\ensuremath{\left({#1}\;\middle|\;{#2}\right)}} 
\newcommand{\@giventhatnostar}[3][]{#1(#2\;#1|\;#3#1)} 
\newcommand{\giventhat}{\@ifstar\@giventhatstar\@giventhatnostar} 
\makeatother

\begin{document}

\title{Betti numbers of random hypersurface arrangements}

\author{Saugata Basu}	
\address{Purdue University, West Lafayette, IN 47906, U.S.A.}
\email{sbasu@math.purdue.edu}

\author{Antonio Lerario}
\address{SISSA (Trieste) and Florida Atlantic University}
\email{lerario@sissa.it}

\author{Abhiram Natarajan}
\address{Purdue University, West Lafayette, IN 47906, U.S.A.}
\email{nataraj2@purdue.edu}

\date{\textbf{\today}} 

\thanks{Saugata Basu was partially supported by NSF grants CCF-1618981, DMS-1620271 and CCF-1910441. Abhiram Natarajan was supported by NSF grant CCF-1910441.}

\begin{abstract}
We study the expected behavior of the Betti numbers of arrangements of the zeros of random (distributed
according to the Kostlan distribution) polynomials in $\RP^n$. Using a random spectral sequence, we prove an asymptotically exact
estimate on the expected number of connected components in the complement of $s$ such hypersurfaces
in $\RP^n$. We also investigate the same problem in the case where the hypersurfaces are defined by random quadratic polynomials. In this case, we establish a connection between the Betti numbers of such arrangements with the expected behavior of a certain model of a randomly defined geometric graph. While our general result implies that the average zeroth Betti number of the union of random hypersurface arrangements is bounded from above by a function that grows linearly in the number of polynomials in the arrangement, using the connection with random graphs, we show an upper bound on the expected zeroth Betti number of random quadrics arrangements that is sublinear in the number of polynomials in the arrangement. This bound is a consequence of a general result on the expected number of connected components in our random graph model which could be of independent interest. 
\end{abstract}

\maketitle 

\section{Introduction}

The quantitative study of the `complexity' of arrangements of  hypersurfaces in some finite dimensional real space has a fairly long history in the area of discrete and computational
geometry (see \cite{Agarwal} for a survey). The main mathematical results concern
the combinatorial, as well as topological, complexities of the so called `cells' of the arrangement.
A cell of an arrangement refers to a connected component of any set obtained as the  intersection of a subset of the given hypersurfaces with the complements of the remaining hypersurfaces (so by definition a cell is always locally closed and a full dimensional cell is open). It is worth recalling some 
of these results.

Given a set of $s$ real algebraic hypersufaces in $\R^n$ each defined by a polynomial of degree
at most $d$, it was proved in \cite{B00} that for each $i, 0 \leq i < n$, the sum over all cells of the
arrangement of the $i$-th Betti number of the cells is bounded from above by $s^{n-i} O(d)^n$. Taking $i=0$, one obtains an upper bound of $s^n O(d)^n$ on the number of cells of the arrangement.

The above results are deterministic. Recently, the study of the \emph{expected} topology
of real varieties or semi-algebraic sets defined by randomly chosen real polynomials
has assumed significance (see for example, \cite{gayet2015expected,FLL,PSC}).
In this paper we initiate the study of quantitative properties of arrangements of 
real hypersurfaces from a random viewpoint in the same spirit as in the papers
referred to above. We study the topological complexity of arrangements of
$s$ randomly chosen hypersurfaces of degrees $d_1,\ldots,d_s$. The probability measure on the space of polynomials, according to which the polynomials are chosen, is the well known Kostlan distribution, which is a Gaussian distribution on the real vector space of homogeneous polynomials of a fixed degree (equipped with an inner product) \cite{EdelmanKostlan95, Kostlan}. Specifically, on the space of homogenous polynomials of degree $d$ in $n+1$ variables, a Kostlan form is defined as
\be
P(x) = \sum_{\substack{(\alpha_0, \ldots, \alpha_n) \\ \sum_{i=0}^n \alpha_i = d}} \xi_{\alpha} x_0^{\alpha_0}\ldots x_n^{\alpha_n},
\ee
where $\xi_{\alpha} \sim \mathcal{N}\left(0, \frac{d!}{\alpha_0! \ldots \alpha_n!}\right)$ are independently chosen. The variances are chosen in such a way that the resulting probability distribution is invariant under an orthogonal change of variables, meaning that there are no preferred points or direction in $\RP^n$, where the zeros of $p$ are naturally defined. Moreover, if we extend this distribution to the space of complex polynomials by replacing real with complex Gaussian variables, it can be shown that this extension is the unique (up to multiples) Gaussian measure which is invariant under unitary change of variables, thus making real Kostlan polynomials a natural object of study.

Here we deviate slightly from the usual convention in the literature in discrete and computational geometry, and consider arrangements of hypersurfaces in real projective space $\RP^n$ rather than in $\R^n$ (since the orthogonal invariance of the Kostlan measure is meaningful only
over the projective space). However, asymptotically it does not make a difference, whether we consider arrangements over affine or projective spaces.  

We consider two variants of the problem of bounding the topological complexity
of an arrangement of random real algebraic hypersurfaces in $\RP^n$ with specified degrees. Our first result outlined in \S \ref{subsec:general} treats the problem in full generality
without any restriction on the degrees (cf. Theorem~\ref{thm:random-arrangement-topology}). We then study the case when all the degrees are assumed to be equal to $2$ (outlined in \S \ref{subsec:quadric}). This is the first
non-trivial case, since for an arrangement of hyperplanes (i.e. with all degrees equal to one), the expected value of the topological
complexity will coincide with that of deterministic generic arrangements. Since, it is known that the growth of the Betti numbers of semi-agebraic sets defined by quadratic polynomials show different behavior compared to that of general semi-algebraic sets 
(see \cite{Bar97, BP'R07jems, Lerario, Basu-Rizzie} for the deterministic case and 
\cite{Letwo, lerario2016gap} in the random setting), it could be expected that the average
topological complexity of arrangements consisting of quadric hypersurfaces would be smaller
than in the general case (at least in the dependence on the number $s$ of hypersurfaces).
We have partial results (outlined in \S \ref{subsec:quadric}) showing that this is indeed the case.
While the $(n-1)$-dimensional Betti number of the complement of a union
$s$ hypersurfaces of degree $d \ge 2$ in $\RP^n$ grows proportionally with $s$ in the deterministic case, we show that in the random case with $d=2$ the expected value of the same is $o(s)$ (cf. Theorem~\ref{thm:quadrics}).

In order to prove Theorem~\ref{thm:quadrics}, we study the behavior 
of a special kind of geometrically defined graph from a random viewpoint (outlined in \S \ref{subsec:graph}).
The geometric graph that we study is a special case of the more general graphs defined 
by semi-algebraic relations which has been widely studied in combinatorics  (see for example \cite{Alon-Pach-et-al}).
In our case the semi-algebraic relation defining the graph is particularly simple and geometric,
and hence we believe that study of
this model could be of interest by itself. We fix a convex semi-algebraic subset  
subset $\mathcal{P}\subset \RP^N$ and sample independent points 
$q_1, \ldots, q_s$ from the uniform distribution on $\RP^N$, and we put an edge between $v_i$ and $v_j$, if and only if $i\neq j$ and the line connecting $q_i$ and $q_j$ \emph{does not} intersect $\mathcal{P}$.
We give a tight estimate on the expected number of isolated points of such a graph (cf. Theorem~\ref{thm:conn-comp-vis-graphintro}), from which we can 
deduce Theorem~\ref{thm:quadrics}. Finally, we conclude by proving a Ramsey-type result about the random graph of quadrics (cf. Corollary \ref{cor:quadrics-graph-ramsey}).
 
\subsection{Random hypersurface arrangements}
\label{subsec:general} 

We are given random homogenous polynomials $P_1, \ldots, P_s$, where each $P_i \in \R[x_0, \ldots, x_n]_{(d_i)}$, and we look at the \emph{random} arrangement of hypersurfaces defined in the projective space by the zero sets of these polynomials, i.e.,
\be\label{eq:Gamma} \Gamma=\bigcup_{j=1}^s\Gamma_j\subset \RP^n,\ee
where each $\Gamma_j$ is  the real algebraic hypersurface given by the zero set of $P_j$, i.e.,
\be \Gamma_j=Z(P_j)=\{[x_0, \ldots, x_n] \in \RP^n \mid P_j(x_0, \ldots, x_n)=0\}.\ee
The main problem that we want to address concerns understanding the topological complexity of $\Gamma$, which will be measured by its Betti numbers\footnote{For a semialgebraic set $S$ we denote by $b_i(S)$ its $i^{\text{th}}$ Betti number with coefficients in $\mathbb{Z}/2\mathbb{Z}$}.

We observe that there are three sets of parameters that will play a role in our study: the degree sequence $d_1, \ldots, d_s$ of the hypersurfaces, the dimension $n$ of the ambient projective space and the number $s$ of independent hypersurfaces.
(Of course, the choice of what is meant by \emph{random} will also play a role: for us the polynomials $P_1, \ldots, P_s$ will be independent samples from the Kostlan ensemble.)

Our first result concerns the asymptotic when $n$ is kept fixed and $d_1, \ldots, d_s, s\to \infty$ and gives information on the number of cells of $\RP^n\backslash \Gamma$. There is clearly an analogous statement for the spherical version of this problem, and the two cases can be related using standard techniques from algebraic topology (the spherical arrangement double covers the projective one and the asymptotics, up to a factor of two, are the same).

\begin{theorem}[$n$ fixed]
\label{thm:random-arrangement-topology}
Let $P_1, \ldots, P_s\in \R[x_0, \ldots, x_n]$ be random, independent, Kostlan polynomials, where $P_i$ has degree $d_i$. Let $\Gamma_i \subset \RP^n$ be the zero set of $P_i$, and define $\Gamma = \bigcup_{i=1}^s \Gamma_i$. Also, let $d = \max{(d_1, \ldots, d_s)}$. Then:
\be
\label{eqn:b0-general} 
\EX{b_0(\RP^n \setminus \Gamma)} = \sum_{\substack{I \subset [s] \\ |I| = n}} \sqrt{\prod_{i \in I} d_i}+O(d^{\nicefrac{(n-1)}{2}}s^{n-1}).
\ee
Moreover if all the degrees are the same $d_1=\cdots=d_s=d$ we have:
\be \label{eq:b0} \EX{b_0(\RP^n \setminus \Gamma)}=\binom{s}{n}d^{\nicefrac{n}{2}}+O(d^{\nicefrac{(n-1)}{2}}s^{n-1}). \ee
\end{theorem}

\begin{remark}As we will prove in Corollary \ref{cor:random-arrangement-topology}, the expectation of the total Betti number of $\RP^n\backslash \Gamma$ has the same order as that of the expected number of connected components (cf. Equation \eqref{eq:b0}). This suggests an interesting phenomenon: the total amount of topology in $\RP^n\backslash \Gamma$ is the same (to the leading order) as the total number of cells of $\RP^n\backslash \Gamma$ and it is therefore natural to conjecture that a random cell is on average homologically a point --- but unfortunately we were not able to prove this result. It is also interesting to compare the previous statement with its worst possible deterministic bound from \cite{BPRcells}:
\be
b_0(\RP^n\backslash \Gamma)\leq {s\choose n}\left(O(d)\right)^{n}.\ee
\end{remark}

\begin{remark}It is possible to produce estimates for the expected number of cells also for other invariant distributions (classified in \cite{Kostlan}), and the answer is given in terms of the parameter of the distribution. In general it is no longer true that we obtain an estimate where the leading term in $d$ is of the type $O(d^{n/2})$, for instance sampling random harmonic polynomials of degree $d$, we get an estimate of the type:
\be\label{eq:harmonic} \EX{b_0(\RP^n \setminus \Gamma)}=\Theta\left(d^ns^n\right).\ee
We sketch a proof of this estimate in Remark \ref{remark:bu} below.
\end{remark}

\subsection{Arrangements of random quadrics} 
\label{subsec:quadric}

%The proof of the previous Theorem uses an argument involving a random spectral sequence, but it gives little information on the top Betti number of $\RP^n\backslash \Gamma$, i.e. on the number of connected components of $\Gamma$ (). 
The, next result deals instead with the asymptotic structure of $\Gamma$ when $d_1, \ldots, d_s = 2$, $n$ is fixed, and $s\to \infty$. It turns out that in this case, the problem of understanding the number of connected components of $\Gamma$, i.e. $b_0(\Gamma)$ (Betti numbers of $\Gamma$ and $\RP^n\backslash \Gamma$ are related by the Alexander-Pontryiagin duality), is related to the connectivity of a certain random graph model, and can be studied in a precise way. Specifically, our second theorem gives an upper bound on the average number of connected components in a random arrangement of quadrics' zero sets.

\begin{theorem}[$n\geq 2$ fixed, $s \to \infty$]
\label{thm:quadrics}
Let $P_1, \ldots, P_s \in \R[X_0, \ldots, X_n]$ be homogeneous Kostlan quadrics, and $n\geq 2$. Let $\Gamma_i \subset \RP^n$ be the zero set of $P_i$, and define $\Gamma = \bigcup_{i=1}^s \Gamma_i$. Then
\be \lim_{s \to \infty} \frac{\EX{b_0(\Gamma)}}{s} = 0.\ee
\end{theorem}
\begin{remark}In the case $n=1$, the expected number of zeroes in $\RP^1$ of a random Kosltan quadric $P$ is $\sqrt{2}$.
Since the probability that two independent quadrics have a zero in common is zero, it follows that, when $n=1$,  $\EX{b_0(\Gamma)}=s\sqrt{2}.$
\end{remark}
\begin{remark}The topology of a random intersection of quadrics has been studied in \cite{Letwo, lerario2016gap}, also using a random spectral sequence (different from the one of this paper). There the following statement is proved: if $X\subset \RP^n$ is an intersection of $k$ random quadrics, then for every fixed $i\geq 0$ with probability that goes to one faster than any polynomial as $n\to \infty$ we have $b_i(X)=1$. In fact this phenomenon follows from a sort of ``rigidification'' of the spectral sequence structure in the large $n$ limit (a similar phenomenon can be observed in the context of this paper).
\end{remark}

As a corollary of Theorem \ref{thm:quadrics} (cf. Corollary \ref{cor:quadrics-graph-ramsey}), we rule out the existence of linear sized cliques in the complement of the quadrics graph. This must be contrasted with a result in \cite{Alon-Pach-et-al} who prove a Ramsey type result (cf. Theorem \ref{thm:alon-semialgebraic-graph}) about existence of sub-linear sized cliques in general semi-algebraic graphs.

\subsection{A random graph model}
\label{subsec:graph}
The result on random arrangements of quadrics unexpectedly follows from the statistic of the number of connected components of a certain random graph introduced as follows. We pick a semialgebraic convex subset $\mathcal{P}\subset \RP^N$ and we sample independent points $q_1, \ldots, q_s$ from the uniform distribution on $\RP^N$. (In the forthcoming connection with the previous problem, $N$ plays the role of the dimension of the space of quadratic forms and the points $q_1, \ldots, q_s$ are the quadrics.) The vertices of the random graph are points $\{v_1, \ldots, v_s\}$ (one for each sample) and we put an edge between $v_i$ and $v_j$, if and only if $i\neq j$ and the line connecting $q_i$ and $q_j$ \emph{does not} intersect $\mathcal{P}$. We call such a graph a \emph{obstacle random graph} and denote it by $\mathcal{G}(\mathcal{P}, s)$. Of course the same definition makes sense in every compact Riemannian manifold, where the notion of convexity comes from geodesics. An obstacle random graph is expected to have at least $s\cdot \frac{\mathrm{vol}(\mathcal{P})}{\mathrm{vol}(\RP^N)}$ many isolated points (this is the expected number of points falling inside $\mathcal{P}$). In Theorem \ref{thm:conn-comp-vis-graphintro} below we prove that to the leading order there are no other isolated points.

\begin{theorem}[$\mathcal{P}\subset \RP^N$ fixed, $s \to \infty$]\label{thm:conn-comp-vis-graphintro}
 The expected number of connected component of the obstacle random graph satisfies
\be \lim_{s\to \infty}\frac{\EX{ b_0(\mathcal{G}(N, \mathcal{P}, s))}}{s}\leq \frac{\vol{\mathcal{P}}}{\vol{\RP^N}}.\ee
\end{theorem}

The connection between Theorem \ref{thm:conn-comp-vis-graphintro} and Theorem \ref{thm:quadrics} comes from an interesting result of Calabi (see Theorem \ref{thm:calabi-characterization} below): the common zero set of two quadrics in $\RP^n$ is nonempty if and only if the line joining these two quadrics (the projective pencil) \emph{does not} intersect the set $\mathcal{P}\subset \RP^N$ of positive quadrics. Since nonempty quadrics in projective space are connected, the incidence graph of the random arrangement $\Gamma=\bigcup_{j=1}^sZ(q_j)$ is the same as the obstacle random graph minus its isolated points coming from vertices $v_i$ whose corresponding quadric $q_i\in \mathcal{P}.$

\subsection{Acknowledgement}
The authors would like to thank Arthur Renaudineau for pointing out a missing hypothesis in the statement of Corollary \ref{cor:random-spectral-ub-lb} in the first arxiv version.

\section{A Random Spectral Sequence}
%%sb hides
\hide{
We direct the reader to references such as~\cite{mccleary2001user} for an in-depth treatment of spectral sequences. We will only consider semi-algebraic sets which are compact, and therefore they posses finite triangulations by \cite[Theorem 9.3.2]{BCR:98}. We shall study the simplicial cohomology (in our case, the topology is tame, so various cohomology theories coincide).
}
%%sb end hide
We will only consider semi-algebraic sets which are compact.  For a compact semi-algebraic set $S$
we denote by $\HH^*(S)$ the cohomology groups of the constant sheaf $(\mathbb{Z}_2)_S$ on $S$. 
Since $S$ is compact $\HH^*(S)$ is isomorphic to the cohomology group $\HH_c^*(S)$ (cohomology of 
$(\mathbb{Z}_2)_S$  with compact support).  Moreover,  since $S$ is semi-algebraic $\HH^*(S)$ is isomorphic to the singular cohomology groups of $S$ as well.      

%%We have a finite family of 
%%sb
%%closed 
Suppose that $\Gamma_1, \ldots, \Gamma_s$ are
compact
semi-algebraic sets  in real projective space,  $\RP^n$.
We want to 
%%consider 
study
the cohomology of the union $\Gamma = \Gamma_1 \cup \cdots \cup \Gamma_s$. The following theorem  follows from the Mayer-Vietoris exact sequence in cohomology for
closed subspaces of a topological space  (see for example,  \cite[page 148]{Iversen}).

\begin{theorem}[Mayer-Vietoris spectral sequence %%(see for e.g. \cite{B00})
]
\label{thm:mayer-vietoris}
There exists a first quadrant cohomological spectral sequence $(E_r^{p,q}, \delta_r: E_r^{p, q} \to E_r^{p+r, q - r + 1})_{r \in \Z_{\ge 0}}$,
with 
\[
E_1^{p,q} = \bigoplus_{\alpha_0 < \ldots < \alpha_{p}}   \HH^q\left( \Gamma_{\alpha_0, \ldots, \alpha_p}\right),
\]
where 
$\Gamma_{\alpha_0, \ldots, \alpha_p} = \Gamma_{\alpha_0} \cap \cdots \cap \Gamma_{\alpha_p}$.

This spectral sequence converges to the cohomology of $\Gamma$,   i.e.
\be
E_{r}^{p, q} \Rightarrow \HH^{p+q}(\Gamma),
\ee
and consequently
\be
\label{eqn:upper-bound-betti-spec}
\text{rank } \HH^i(\Gamma ) = \sum_{p+q = i} \text{rank } {E}_{\infty}^{p,q}.
\ee
Also, this spectral sequence collapses at $E_n$,   and hence
\be
E_{\infty}^{n-1, 0} \cong E_{n}^{n-1, 0}.
\ee
\end{theorem}

\begin{remark}
An alternative way to obtain the same spectral sequence in 
Theorem~\ref{thm:mayer-vietoris} but using open covers instead of closed covers is as follows. 
It follows from the
conic structure theorem at infinity of semi-algebraic sets
(see for example \cite[Proposition 5.49]{BPRbook2}),
i.e. that there exist
open semi-algebraic neighborhoods $\Gamma_1^+,\ldots,\Gamma_s^+$ of $\Gamma_1, \ldots,\Gamma_s$ respectively,
such that for any $1 \leq \alpha_0 < \cdots < \alpha_p \leq s$, 
$\Gamma^+_{\alpha_0} \cap \cdots \cap \Gamma^+_{\alpha_p}$ is homotopy equivalent to 
$\Gamma_{\alpha_0} \cap \cdots \cap \Gamma_{\alpha_p}$,
and $\Gamma^+ = \Gamma^+_1 \cup \cdots \cup \Gamma^+_s$ is homotopy equivalent to
$\Gamma = \Gamma_1 \cup \cdots \cup \Gamma_s$.  
Then the spectral sequence in Theorem~\ref{thm:mayer-vietoris} is isomorphic to the spectral sequence associated to the open cover of 
$\Gamma^+$ by the subsets $\Gamma_1^+,\ldots,\Gamma_s^+$ (see for example 
\cite[Ex.  15.7.1]{botttu1982}).  
\end{remark}
 
%%sb hides
\hide{
Let $A_1, \ldots, A_s$ be triangulations of $\Gamma_1, \ldots, \Gamma_s$, respectively. Thus we have a finite simplicial complex $A = A_1 \cup \ldots \cup A_s$. By definition, any finite intersection $A_{\alpha_0} \cap \ldots \cap A_{\alpha_p}$, denoted $A_{\alpha_0, \ldots, \alpha_p}$, is a sub-complex of $A$. Let $C^{i}(A)$ denote the vector space over $\R$ of $i$-co-chains of $A$, and $C^{*}(A) = \bigoplus_i C^i(A)$. We shall use the Mayer-Vietoris spectral sequence.

\begin{theorem}[Mayer-Vietoris spectral sequence (see for e.g. \cite{B00})]
\label{thm:mayer-vietoris}
There exists a first quadrant cohomological spectral sequence $(E_r, \delta_r)_{r \in \Z_{\ge 0}}$, where each $E_r$ is a double complex
\be
E_r = \bigoplus_{p, q \in \Z_{\ge 0}} E_r^{p, q},
\ee
and
\be
E_0^{p, q} = \bigoplus_{\alpha_0 < \ldots < \alpha_{p}} C^q(A_{\alpha_0, \ldots, \alpha_p}),
\ee
with morphisms
\be
\delta_r: E_r^{p, q} \to E_r^{p+r, q - r + 1},
\ee
where
\be
E_{r+1} \cong H_{\delta_{r}}(E_{r}).
\ee
This spectral sequence converges to the cohomology of the union, i.e.
\be
E_{r}^{p, q} \Rightarrow \HH^{p+q}(A),
\ee
and consequently
\be
\label{eqn:upper-bound-betti-spec}
\text{rank } \HH^i(A) = \sum_{p+q = i} \text{rank } {E}_{\infty}^{p,q}.
\ee
Also, this spectral sequence collapses at $E_n$, i.e.
\be
E_{\infty}^{n-1, 0} \cong E_{n}^{n-1, 0}.
\ee
\end{theorem}
}

\begin{corollary}[of Theorem \ref{thm:mayer-vietoris}]
\label{cor:random-spectral-ub-lb}
%%Let $A_1, \ldots, A_s$ be random simplicial complexes. 
Let $\Gamma_1, \ldots, \Gamma_s$ be randomly chosen closed subspaces of
$\RP^n$. Consider the same definitions as in Theorem \ref{thm:mayer-vietoris}. For every $r \in \Z_{\ge 0}$, define $e_r^{p, q} := \EX{\text{rank } E_r^{p, q}}$. 
We have
\be
\label{eqn:e-infty-ub}
e_{r+1}^{p, q} \le e_{r}^{p, q}, 
\ee
and, when $E_r^{p + r, q - r + 1} = 0$,
\be
\label{eqn:e-lower-bound}
e_{r+1}^{p,q} \ge e_{r}^{p,q} - e_{r}^{p-r, q+r-1}.
\ee
\end{corollary}

\begin{proof}
Follows immediately from the deterministic versions of the same statements, which in turn follow from the structure of the differentials, i.e., specifically the fact that
\be
E_{r+1}^{p, q} \cong \bigslant{\Ker{\delta_r: E_r^{p, q} \to E_r^{p + r, q - r + 1}}}{\Img{\delta_r: E_r^{p - r, q + r - 1} \to E_r^{p, q}}}.
\ee
\end{proof}

\subsection{Average Betti numbers of hypersurface arrangements}

\begin{proof}[Proof of Theorem \ref{thm:random-arrangement-topology}]
We give the proof for the spherical case; the asymptotic for projective case needs to be divided by two. 
By Theorem \ref{thm:mayer-vietoris}, 
\be
%%E_1^{p, q} \cong \bigoplus_{\alpha_0 < \ldots < \alpha_{p}} H^q(A_{\alpha_0, \ldots, \alpha_p}).
E_1^{p, q} \cong \bigoplus_{\alpha_0 < \ldots < \alpha_{p}} \HH^q(\Gamma_{\alpha_0, \ldots, \alpha_p}).
\ee

%%In our case, we have random complexes $A_{\alpha_0, \ldots, \alpha_p}$, and we need two results. 
In our case, we have random closed subsets  $\Gamma_{\alpha_0, \ldots, \alpha_p}$,  and we need two results. 

First let us recall the Edelman-Kostlan-Shub-Smale Theorem \cite{edelman1995many,Kostlan, shub1993complexity}, on the expected number of common solutions in $\RP^n$ of a random system of equations: if $P_1, \ldots, P_n$ are random, independent, homogeneous polynomials of degrees $d_1, \ldots, d_n$, Kostlan distributed, then the expected number of common zeroes in $\RP^n$ of these polynomials equals $\sqrt{d_1\cdots d_n}$. 

In particular, since 
%%$A_{\alpha_0, \ldots,  \alpha_{n-1}}$
$\Gamma_{\alpha_0, \ldots,  \alpha_{n-1}}$
 is the intersection in $\RP^n$ of $n$ random, independent, Kostlan distributed hypersurfaces of degree $d_{\alpha_0}, \ldots, d_{\alpha_{n-1}}$, the previous result gives the precise value of the expected rank of 
 %%$\bigoplus_{\alpha_0 < \ldots < \alpha_{n-1}} H^0(A_{\alpha_0, \ldots, \alpha_{n-1}})$:
 $\bigoplus_{\alpha_0 < \ldots < \alpha_{n-1}} \HH^0(\Gamma_{\alpha_0, \ldots, \alpha_{n-1}})$:
\be
\label{eqn:EKSS}
e_1^{n-1,0} = 2\sum_{\substack{I \subset [s] \\ |I| = n}} \sqrt{\prod_{i \in I} d_i}.
\ee

Next, we need a result of Gayet-Welschinger~\cite{gayet2015expected}, which states that for every $0\leq i\leq m$ the expectation of the $i$--th Betti number of a random submanifold  $X\subset\RP^n$ defined by $n-m$ independen Kostlan polynomials, is bounded by $\EX{b_i( X)}\leq O(\sqrt{d^m}).$ 

This implies immediately that
\be
\label{eqn:gayet-welschinger}
e_1^{p, q} \le \binom{s}{p+1}O(d^{\nicefrac{(n-p-1)}{2}}),
\ee
for any $p < n-1$. In fact, each summand in $e_1^{p,q}$ is the expected $q$--th Betti number of a random intersection of $p+1$ random Kostlan hyeprsurfaces and,  since there are a total of $\binom{s}{p+1}$ such summands, then \eqref{eqn:gayet-welschinger} follows.

%% Antonio's addition

Denoting the reduced Betti numbers of a manifold by $\tilde{b}_*(\cdot)$, by the Alexander-Pontryiagin duality, we have
\be
b_{n-1}(\Gamma) = \tilde{b}_{n-1}(\Gamma) = \tilde{b}_0(S^n \setminus \Gamma),
\ee
thus when $n > 1$,
\be
\label{eqn:alex-pon-1}
b_{n-1}(\Gamma) = \tilde{b}_0(S^n \setminus \Gamma) = b_0(S^n \setminus \Gamma) - 1,
\ee
and when $n \le 1$,
\be
\label{eqn:alex-pon-2}
b_{n-1}(\Gamma) = \tilde{b}_0(S^n \setminus \Gamma) = b_0(S^n \setminus \Gamma).
\ee
Consequently, when $n > 1$,
\begin{align}
\EX{b_0(S^n \setminus \Gamma)} &= \EX{b_{n-1}(\Gamma)} + 1 \eqcomment{by \eqref{eqn:alex-pon-1}} \\
&= \sum_{k=1}^n e_{\infty}^{n-k, k-1} + 1 \eqcomment{by \eqref{eqn:upper-bound-betti-spec} and linearity of expectation}, \label{eqn:b0-complement-1}
\end{align}
and when $n \le 1$,
\begin{align}
\EX{b_0(S^n \setminus \Gamma)} &= \EX{b_{n-1}(\Gamma)} \eqcomment{by \eqref{eqn:alex-pon-2}} \\
&= \sum_{k=1}^n e_{\infty}^{n-k, k-1}. \eqcomment{by \eqref{eqn:upper-bound-betti-spec} and linearity of expectation}. \label{eqn:b0-complement-2}
\end{align}
First, observe that 
\begin{align}
\sum_{k = 2}^n e_{\infty}^{n-k, k-1} &\le \sum_{k = 2}^n e_{1}^{n-k, k-1} \eqcomment{by \eqref{eqn:e-infty-ub}}\\
&\le \sum_{k = 2}^n \binom{s}{n-k+1}O(d^{\nicefrac{(k-1)}{2}}) \eqcomment{by \eqref{eqn:gayet-welschinger}} \\
&\le s^{n-1}O(d^{\nicefrac{(n-1)}{2}}). \label{eqn:other-terms-ub}
\end{align}
Now it remains to give precise bounds on $e_{\infty}^{n-1, 0}$, which is the same as as obtaining precise bounds on $e_n^{n-1, 0}$, given that the spectral sequence collapses at $E_n$ (cf. Theorem \ref{thm:mayer-vietoris}). Thus,
\be
\label{eqn:bottom-term-ub}
e_{\infty}^{n-1, 0} = e_n^{n-1, 0} \le e_1^{n-1,0} = 2\sum_{\substack{I \subset [s] \\ |I| = n}} \sqrt{\prod_{i \in I} d_i},
\ee
where the second step is due to \eqref{eqn:e-infty-ub} and the final equality is by \eqref{eqn:EKSS}. This provides an upper bound on $e_{\infty}^{n-1, 0}$. To lower bound $e_{\infty}^{n-1, 0}$, we make repeated use of inequality \eqref{eqn:e-lower-bound}. Note that \eqref{eqn:e-lower-bound} is true only when $E_r^{p + r, q - r + 1} = 0$, which happens when $r > q + 1$ since we have a first-quadrant spectral sequence. Thus \eqref{eqn:e-lower-bound} holds for all $e_r^{p, 0}$ when $r > 1$. Telescoping, we get
\begin{align}
e_{\infty}^{n-1, 0} = e_n^{n-1, 0} &\ge e_{n-1}^{n-1, 0} - e_{n-1}^{0, n-2} \eqcomment{by \eqref{eqn:e-lower-bound}} \\
&\ge e_{n-1}^{n-1, 0} - e_1^{0, n-2} \eqcomment{by \eqref{eqn:e-infty-ub}} \\
&\ge e_{n-2}^{n-1, 0} - e_{n-2}^{1, n-3} - e_1^{0, n-2} \eqcomment{by \eqref{eqn:e-lower-bound}} \\
&\ge e_{n-2}^{n-1, 0} - e_{1}^{1, n-3} - e_1^{0, n-2} \eqcomment{by \eqref{eqn:e-infty-ub}} \\
&\vdots \\
&\ge e_1^{n-1, 0} - \left(\sum_{i=0}^{n - 2} e_1^{i,n-2-i}\right) \\
&\ge 2\sum_{\substack{I \subset [s] \\ |I| = n}} \sqrt{\prod_{i \in I} d_i} - \left(\sum_{i=0}^{n - 2} \binom{s}{i+1} O\left(d^{\nicefrac{(n-i-1)}{2}}\right)\right) \eqcomment{by \eqref{eqn:EKSS} and \eqref{eqn:gayet-welschinger}}. \label{eqn:bottom-term-lb}
\end{align}
Altogether,
\be
\label{eqn:bottom-term-bounds}
2\sum_{\substack{I \subset [s] \\ |I| = n}} \sqrt{\prod_{i \in I} d_i} \ge e_{\infty}^{n-1, 0} \ge 2\sum_{\substack{I \subset [s] \\ |I| = n}} \sqrt{\prod_{i \in I} d_i} - s^{n-1}O\left(d^{\nicefrac{(n-1)}{2}}\right).
\ee
From Equations \eqref{eqn:bottom-term-bounds} and \eqref{eqn:other-terms-ub}, we get that
\be
\label{eqn:e-infty-sum}
\sum_{k=1}^n e_{\infty}^{n-k, k-1} = 2\sum_{\substack{I \subset [s] \\ |I| = n}} \sqrt{\prod_{i \in I} d_i} + s^{n-1}O\left(d^{\nicefrac{(n-1)}{2}}\right).
\ee
Putting \eqref{eqn:e-infty-sum} in Equations \eqref{eqn:b0-complement-1} and \eqref{eqn:b0-complement-2} gives us that for any $n$,
\be
\EX{b_0(S^n \setminus \Gamma)} = 2\sum_{\substack{I \subset [s] \\ |I| = n}} \sqrt{\prod_{i \in I} d_i} + s^{n-1}O\left(d^{\nicefrac{(n-1)}{2}}\right),
\ee
completing the proof of the theorem.
\end{proof}
\begin{remark}[Sketch of the proof of \eqref{eq:harmonic}]\label{remark:bu}Following the same lines of the previous proof, in the case of random harmonic polynomials, we have
\be
\label{eqn:har1}
e_1^{n-1,0} = 2\sum_{\substack{I \subset [s] \\ |I| = n}}\Theta\left( d^n\right)=\Theta\left(d^ns^n\right).
\ee
In fact in this case one can use a result of B\"urgisser \cite{Bueuler} for estimating the expectation of the Euler characteristic of a random intersection 
%%$A_{\alpha_0, \ldots,  \alpha_{n-1}}$ 
$\Gamma_{\alpha_0, \ldots,  \alpha_{n-1}}$ 
in $\RP^n$ of $n$ random, independent, harmonic hypersurfaces of degree $d_{\alpha_0}, \ldots, d_{\alpha_{n-1}}$. In this case the intersection is almost surely zero dimensional, and the Euler characteristic coincides with the number of points. Using the notation of \cite{Bueuler}, the parameter $\delta$ of a random harmonic polynomial of degree $d$ is $\delta=\Theta(d)$ (see \cite[Lemma 16]{FLL} and \cite[Example 2]{FLL}) and \cite[Theorem 1.1]{Bueuler} implies \eqref{eqn:har1}. 
For the analog of \eqref{eqn:gayet-welschinger} in the harmonic case we can use Thom-Milnor's bound $b(X)\leq O(d^n)$ for the smooth zero set $X$ of an algebraic set in $\mathbb{R}\mathrm{P}^n$. This gives now:
\be \label{eq:har2}e_1^{p, q} \le \binom{s}{p+1}O(d^{n}).
\ee
Since in \eqref{eq:har2} we have $p\geq1$, the estimate \eqref{eqn:har1} dominates \eqref{eq:har2} and \eqref{eq:harmonic} follows.

Strictly speaking we observe that the bound that we have used $b(X)\leq O(d^n)$, as in Milnor's proof, would be of the order $O(d^{n})$ in the case of an affine algebraic set defined by polynomials of degree bounded by $d$, and of the order $O(d^{n+1})$ in the case of a projective algebraic set (which is the case of our interest), but Milnor's proof can be easily improved as follows. 

\begin{lemma}For every $n\geq 1$ there exists $c_n$ such that for every algebraic set $X\subset \mathbb{R}\mathrm{P}^n$ defined by polynomials of degree bounded by $d$, we have  $b(X)\leq c_n d^n$.
\end{lemma}
\begin{proof} We show it by induction on $n$. 

In the case $n=1$, if $X\neq \mathbb{R}\mathrm{P}^1$, then $X$ consists of finitely many points and, up to a change of coordinates, we can assume all this points are contained in the open set $\{x_0\neq 0\}\simeq \mathbb{R}$. Therefore $b(X)=\#X\leq d.$

Let now $n>1$. Consider in $\mathbb{R}\mathrm{P}^n$ the open sets $A_\epsilon=\{x_0^2<\epsilon(x_0^2+\cdots +x_n^2)\}$ and $B=\{x_0\neq 0\}.$ Then, by Mayer-Vietoris,
$$b(X)\leq b(A_\epsilon\cap X)+b(B\cap X)+b(A_\epsilon\cap B\cap X).$$
For $\epsilon>0$ small enough we have the following:
\begin{enumerate}
\item  set $A_\epsilon\cap X$ deformation retracts onto $X\cap \{x_0=0\}=X\cap \mathbb{R}\mathrm{P}^{n-1}$. Therefore, by the inductive hypothesis, $b(X\cap A_\epsilon)\leq c_{n-1}d^{n-1};$
\item the set $A_\epsilon\cap B\cap X$ deformation retracts onto $X\cap \{x_0^2=\epsilon(x_0^2+\cdots +x_n^2)\}.$ This is an algebraic set in $\{x_0\neq 0\}\simeq \mathbb{R}^n$ defined by polynomials of degree bounded by $d$ and, by the affine Milnor's bound, $b(A_\epsilon\cap B\cap X)\leq O(d^{n});$ 
\item the set $B\cap X$ is an algebraic set in $\{x_0\neq 0\}\simeq \mathbb{R}^n$ defined by polynomials of degree bounded by $d$ and, again by the affine Milnor's bound, $b(B\cap X)\leq O(d^n).$
\end{enumerate} 
Putting these three estimates together, we get the statement.
\end{proof}
 \end{remark}
Below we give a corollary of Theorem~\ref{thm:random-arrangement-topology} which gives a bound on the sum of the Betti numbers of $\RP^n \setminus \Gamma$ (we prove the corollary for the spherical case, again one has to divide the asymptotics by two in the projective case).
\begin{corollary}[$n$ fixed]
\label{cor:random-arrangement-topology}
Let $\Gamma$ be defined as in Theorem \ref{thm:random-arrangement-topology}. Then, for all $k > 0$,
\be 
\label{eqn:higher-betti-numbers}
\EX{b_k(S^n \setminus \Gamma)} = O(d^{\nicefrac{(n-1)}{2}}s^{n-k}).
\ee
Consequently,
\be 
\label{eqn:betti-numbers-sum}
\EX{\sum_{i=0}^{n-1} b_i(S^n \setminus \Gamma)} = 2\sum_{\substack{I \subset [s] \\ |I| = n}} \sqrt{\prod_{i \in I} d_i}+  O(d^{\nicefrac{(n-1)}{2}}s^{n-1}).
\ee
\end{corollary}
\begin{proof}
By Alexander-Pontryiagin duality, when $k > 0$,
\be
b_k(S^n \setminus \Gamma) = \tilde{b}_k(S^n \setminus \Gamma) = \tilde{b}_{n-k-1}(\Gamma) \le b_{n-k-1}(\Gamma),
\ee
thus
\begin{align}
\EX{b_k(S^n \setminus \Gamma)} &\le \EX{b_{n-k-1}(\Gamma)} \\
&= \sum_{i=0}^{n-k-1} e_{\infty}^{i, n-k-i-1} \\
&\le \sum_{i=0}^{n-k-1} e_{1}^{i, n-k-i-1} \\
&\le \sum_{i=0}^{n-k-1} \binom{s}{i+1}O\left(d^{\nicefrac{(n-i-1)}{2}}\right) \eqcomment{by \eqref{eqn:gayet-welschinger}}\\
&\le s^{n-k}O\left(d^{\nicefrac{(n-1)}{2}}\right),
\end{align}
proving Equation \eqref{eqn:higher-betti-numbers}. Using this, Equation \eqref{eqn:b0-general} of Theorem \ref{thm:random-arrangement-topology}, and linearity of expectation, Equation \eqref{eqn:betti-numbers-sum} follows immediately.
\end{proof}
Thus the expected total Betti number of $\RP^n\backslash \Gamma$ has the same order as that of its number of connected components.

\section{Obstacle Random Graphs and an Application to Arrangement of Quadrics}

%Theorem~\ref{thm:random-arrangement-topology} gives no information on the top Betti number of $\RP^n\backslash \Gamma$, i.e. on the number of connected components of $\Gamma$. 

In this section, we study the top Betti number of $\RP^n \setminus \Gamma$, when $\Gamma$ is the union of a finite set of quadrics. It turns out that in this case, the problem of understanding the number of connected components of $\Gamma$ is related to the connectivity of a certain random graph model. 

In the study of the topological complexity of arrangements of hypersurfaces, there are two sets of parameters that play a part. First is the sequence of degrees of the polynomials defining the hypersurfaces. Second is the number of polynomials in the arrangement. The former is often called the `algebraic part' and the latter is called the `combinatorial part'. While the algebraic part is indeed important, in several applications, for instance in discrete and computational geometry, it is the combinatorial part of the complexity that is of paramount interest. This is because one typically encounters arrangements of a large number of objects, where each object has ``bounded complexity''. 

Theorem \ref{thm:random-arrangement-topology} and Corollary \ref{cor:random-arrangement-topology} together suggest that in arrangements of $s$ random hypersurfaces, the top Betti number of the complement of the union of the arrangement grows linearly in $s$. In line with many results where the growth of the Betti numbers of semi-algebraic sets defined by quadratic inequalities is shown to be different, in this section we prove a bound on the average top Betti number of the complement of the union of an arrangement of Kostlan quadrics that is sub-linear in $s$. In Section \ref{sec:random-graph-model}, we introduce our random graph model which we call ``Obstacle'' random graphs. In Section \ref{sec:avg-conn-comp}, we prove a theorem (Theorem~\ref{thm:conn-comp-vis-graph}) about the average number of connected components in this random graph model. Then, in Section \ref{sec:b0-quad}, using a theorem of Calabi (Theorem~\ref{thm:calabi-characterization}), we obtain a result on the average zeroth Betti number of $\Gamma$ (Theorem \ref{thm:quadrics}), when $\Gamma$ is a finite union of the zero sets of quadrics.

\subsection{The `Obstacle' random graph model}
\label{sec:random-graph-model}

In this section we introduce the obstacle random graph model. Before, we need a definition of convexity in $\RP^n.$ 
\begin{definition}[Convex set in $\RP^N$] Let $\Pos\subset \RP^N$ be a measurable set and $f:S^N\to \RP^N$ the double cover map. We will say that $\Pos$ is convex if it is entirely contained in one single affine chart \emph{and}  each one of the two components of $f^{-1}(\Pos)$ is geodesically convex in  $S^N$.
\end{definition}
Note that convex sets in $\RP^N$ are contractible and that for every pair of points $q_0, q_1\in \Pos$ there is a ``segment'' joining these two points and entirely contained in $\Pos$. This segment is build as follows: first lift $q_0, q_1$ to two points $\tilde q_0, \tilde q_1\in S^N$ both belonging to the same component $\tilde\Pos$ of $p^{-1}(\Pos)=\tilde\Pos\sqcup -\tilde\Pos$. Then consider the spherical arc
\begin{equation}
\tilde q_t:=\frac{(1-t)\tilde q_0+t\tilde q_1}{\|(1-t)\tilde q_0+t\tilde q_1\|}, \quad t\in [0, 1].
\end{equation}
This spherical arc is a (reparametrized) geodesic joining $\tilde q_0$ to $\tilde q_1$ and therefore, by assumption, it lies in $\tilde\Pos$. The arc joining $q_0$ with $q_1$ is $q_t:=p(\tilde q_t)$, $t\in [0, 1]$.

Notice that with this definition a ball $B_{\RP^N}(x, \frac{\pi}{2}-\epsilon)$, for $\epsilon>0$ small enough, is convex but not geodesically convex in $\RP^N$. 
\begin{definition}[`Obstacle' random graph] 
\label{def:obstable-random-graph}
Let $\{q_1, \ldots, q_s\} \subset \RP^N$ be a sample from the uniform distribution on $\RP^N$, and let $\Pos \subset \RP^N$ (the ``obstacle'') be a measurable convex set. We define the obstacle random graph model $\graph(N, \Pos, s)$ as follows:
\begin{enumerate}[1.]
\item $\graph(N, \Pos, s)$ has $s$ vertices $\{q_1, \ldots, q_s\}$.
\item Define $\ell(q_i, q_j) := \{[\lambda_a q_i+\lambda_b q_j]\}_{[\lambda_a, \lambda_b]\in \RP^1}$. The edge set is defined as the set of unordered pairs
\be
\left\{(q_i, q_j) \mid 1 \le i < j \le s \text{ and } \ell(q_i, q_j) \cap \Pos = \emptyset\right\}.
\ee
\end{enumerate}
In other words, it is an undirected graph where the vertices are $\{q_1, \ldots, q_s\}$, and for every pair of distinct vertices $q_i, q_j$ has an undirected edge if and only if the great circle connecting the vertices \emph{does not} intersect $\Pos$.
\end{definition}

This model bears some similarity to random visibility graphs \cite{de2000visibility}. See Figure~\ref{fig:obstacle-random-graph} for an example illustration.

\begin{illustration}
\begin{figure}
\label{fig:obstacle-random-graph}
\vspace{5pt}
\begin{tikzpicture}[scale=1.7]
% Draw Sphere
       \pgfmathsetmacro\R{2} 
	   \fill[ball color=\spherecolor,opacity=\sphereopacity] (0,0) circle (\R); % 3D lighting effect
        \node[above,color=black]  at (0,1.6) {\large $\RP^N$};

% Cone of Positive Quadratic Forms
		\fill[fill=\pncolor,opacity=\pnopacity] {[rotate=0] (48:2 and 0.9) arc (48:100:2 and 0.9)} {[rotate=320] arc (115:73.5:2 and 0.5)} {[rotate=220] arc(106.5:136.5:2 and 0.5)}-- cycle;
        \node[above,color=black]  at (1.0,0.35) {\large $\mathcal{P}$};

% Vertices of Graph
          \begin{scope}[rotate=60]
		    \fill(-0.52,0.97) circle (1pt);
          \end{scope}
          \begin{scope}[rotate=150]
		    \fill(-0.85,-1.36) circle (1pt);
          \end{scope}
          \begin{scope}[rotate=120]
		    \fill(-0.68,-1.07) circle (1pt);
		    \fill(1.29,-0.87) circle (1pt);
          \end{scope}
          \begin{scope}[rotate=90]
		    \fill(-1.41,-1.09) circle (1pt);
          \end{scope}
		  \vertex{2}{0.45}{76}{320}
		  \vertex{2}{0.45}{105}{320}

% Edges of Graph
          \begin{scope}[rotate=356.5]
            \path[draw,thick,\edgecolor] (-1.12,0) arc [start angle=236,
              end angle=310,
              x radius=2cm,
            y radius=0.65cm] ;
          \end{scope}
          \begin{scope}[rotate=60]
            \path[draw,thick,\edgecolor] (1.41,0.70) arc [start angle=45,
              end angle=105,
              x radius=2cm,
            y radius=1cm] ;
          \end{scope}
          \begin{scope}[rotate=150]
            \path[draw,thick,\edgecolor] (-0.85,-1.36) arc [start angle=245,
              end angle=290,
              x radius=2cm,
            y radius=1.5cm] ;
          \end{scope}
          \begin{scope}[rotate=62]
            \path[draw,thick,\edgecolor] (0.52,-1.16) arc [start angle=285,
              end angle=311,
              x radius=2cm,
            y radius=1.2cm] ;
          \end{scope}
          \begin{scope}[rotate=90]
            \path[draw,thick,\edgecolor] (-1.41,-1.09) arc [start angle=225,
              end angle=292,
              x radius=2cm,
            y radius=1.54cm] ;
          \end{scope}
          \begin{scope}[rotate=320]
            \path[draw,thick,\edgecolor] (-0.85,-0.68) arc [start angle=245,
              end angle=332,
              x radius=2cm,
            y radius=0.75cm] ;
          \end{scope}
		  \edgeA{76}{105}{2}{0.45}{320}

% Non-Edges of the Graph
		\nonedgeA{217}{260}{2}{1.22}{103}
		\nonedgeA{239}{322}{2}{-0.44}{18.5}
		\nonedgeA{213}{315}{2}{0.95}{111}
		\nonedgeA{250}{310}{2}{1.14}{120}
		\nonedgeA{38}{72}{2}{0.24}{81}
		\nonedgeA{39}{90}{2}{0.68}{12.5}
		\nonedgeA{54}{100}{2}{0.58}{336}
		\nonedgeA{27}{108}{2}{0.28}{303}
		\nonedgeA{83}{118}{2}{0.65}{30}
		\nonedgeA{87}{140}{2}{0.65}{282.5}
		\nonedgeA{286}{323}{2}{0.37}{36}
		\nonedgeA{51}{72}{2}{0.11}{354}
		\nonedgeA{208}{262}{2}{0.6}{118}
		\nonedgeA{70}{125}{2}{0.03}{0}
\end{tikzpicture}
\caption{Illustration of obstacle random graph. The thick lines denote edges of the graph, while the dotted lines denote non-edges, i.e. edges that were not included in the random graph because their geodesic completion intersected $\Pos$.}
\end{figure}
\end{illustration}

\begin{remark}
Two commonly studied random graph models are the Erd{\"o}s-R{\'e}nyi model (proposed in \cite{erdos1959random,gilbert1959random}), and the geometric random graph model (proposed in \cite{gilbert1961random}).

\begin{itemize}
\item In the obstacle random graph, the edge probabilities are random variables, and the random variables are not independent. Thus this model is dissimilar to the Erd{\"o}s-R{\'e}nyi model.
\item Define the metric $d: \RP^N \times \RP^N \to \R,$ where
  \be
    d(q, q') = \begin{cases} 0 & q_1 = q_2 \\ 
                                  1 & \ell(q_1, q_2) \cap \Pos \neq \emptyset \\
									\frac{1}{2} & \mathrm{otherwise}
              \end{cases},
  \ee
  where $\ell(q_1, q_2)$, as defined earlier, is the projective line containing the points $q_1$ and $q_2$. While our graph is a geometric random graph on $s$ vertices with an edge appearing between two distinct vertices $q_1, q_2$ when $d(q_1, q_2) \le \frac{1}{2}$, note that $d$ is a non-continuous function that is difficult to work with, and thus standard results in the geometric random graph literature do not apply. \end{itemize}
\end{remark}

\subsection{$b_0$ of arrangement of random quadrics}
\label{sec:b0-quad}

Our theorem on the expected $b_0$ of an arrangement of random quadrics, i.e. Theorem \ref{thm:quadrics}, follows from a general theorem that we prove about the average number of connected components in the obstacle random graph model $\graph(N, \Pos, s)$ as $s \to \infty$.

\begin{theorem}[$N$ fixed]
\label{thm:conn-comp-vis-graph}
Consider the obstacle random graph model $\graph(N, \Pos, s)$ as per Definition \ref{def:obstable-random-graph}. Then
\be
\lim_{s \to \infty} \frac{\EX{b_0(\graph(N, \Pos, s))}}{s} \le \frac{\vol{\mathcal{P}}}{\vol{\RP^N}}.	
\ee
\end{theorem}

The proof of Theorem \ref{thm:conn-comp-vis-graph} is deferred to Section \ref{sec:avg-conn-comp}. 

Once we fix a scalar product on $\R^{n+1}$, there is a natural isomorphism between the vector space $\textrm{Sym}(n+1, \R)$ of real symmetric matrices and the space $\R[x_0, \ldots, x_n]_{(2)}$, which is given by associating to a symmetric matrix $Q$ the quadratic form defined by $q(x)=\langle x, Qx\rangle.$ It turns out that the Kostlan measure is the pushforward of the GOE\footnote{Stands for Gaussian Orthogonal Ensemble (see \cite{tao2012topics} for a description).} measure under this linear isomorphism (see for e.g. \cite{lerario2016gap} for a discussion about this), i.e.:
\be \textrm{$Q$ is a GOE matrix} \iff \textrm{$q$ is a Kostlan polynomial}.\ee

Let $\RP^N=P(\textrm{Sym}(n, \R))$ be the projectivization of the space of symmetric matrices (here $N={n+2\choose 2}-1$) and consider the set $\mathcal{P}_n\subset \RP^N$ which is the projectivization of the set of positive definite matrices (equivalently of the set of \emph{positive} quadratic forms):
\be \mathcal{P}_n=\{[Q]\in \RP^N\,|\, Q>0\}.\ee
We endow $\textrm{Sym}(n+1, \R)$ with the Frobenius metric (which corresponds to the Bombieri-Weil metric under the above linear isomorphism); on the projective space $\RP^N$ we consider the quotient Riemannian metric (for this metric the quotient map $p:S^N\to \RP^N$ is a local isometry), with corresponding volume density. In this way, if $q$ is a random Kostlan quadric, we have:
\be \label{eqn:prob-quadric-positive}\PP\{\textrm{$q$ is a positive form}\}=\frac{\textrm{vol}(\mathcal{P}_n)}{\textrm{vol}(\RP^N)}.\ee

\begin{remark}
The relative volume of $\mathcal{P}_n$ in $\RP^N$ is known (see e.g. \cite{majumdar2011many}) to decay exponentially fast when $n$ increases:
\be 
\label{eqn:relative-vol-postive-definite-cone}
\lim_{n\to \infty}\frac{1}{n^2}\log\left(\frac{\textrm{vol}(\mathcal{P}_n)}{\textrm{vol}(\RP^N)}\right)=-\frac{\log3}{4}.
\ee
\end{remark}

The following result, which is due to Calabi \cite{calabi1964}, gives a geometric criterion for two quadrics intersecting in projective space.

\begin{theorem}[Calabi, 1964]
\label{thm:calabi-characterization}
For $n \geq 2$ let $q_1, q_2\in \R[x_0, \ldots, x_n]_{(2)}$ and denote by  $\Gamma_1,\Gamma_2\subset \RP^n$ their (possibly empty) zero sets. Define $\ell(q_1, q_2) \subset \RP^N$ to be the projective line $\ell(q_1, q_2) := \{[\lambda_1 q_1+\lambda_2 q_2]\}_{[\lambda_1, \lambda_2]\in \RP^1}$ (a pencil of quadrics). Then:
\begin{equation}
\Gamma_1 \cap \Gamma_2 \neq \emptyset \iff \ell \cap \mathcal{P}_n = \emptyset.
\end{equation}
\end{theorem}

One can refer to \cite{lerario2012convex} for a proof of this using spectral sequences. As a consequence of Calabi's Theorem (Theorem~\ref{thm:calabi-characterization}), studying the average zeroth Betti number of $\Gamma$ reduces to studying the average number of connected components in the ostacle random graph model, i.e. studying the average number of connected components of $\graph(N, \Pos_n, s)$. 

Specifically, nonempty quadrics in projective space are connected, therefore the number of connected components of $\Gamma$ equals the number of connected components of the incidence graph of the zero sets $Z(q_i)$ of the sampled quadrics. This incidence graph is a subgraph of the corresponding obstacle graph -- to form the subgraph, we must discard the points that fall inside $\Pos_n$ because the zero sets of quadrics in $\Pos_n$ is empty. 

\begin{definition}
\label{defn:quadrics-graph}
We define the `quadrics graph' as the graph form by taking an instance of $\graph(N, \Pos_n, s)$ and forming a subgraph of it by discarding the vertices that fall inside $\Pos_n$. We shall use $\mathcal{H}(N, \Pos_n, s)$ to denote this random subgraph model.
\end{definition}

Relying on Calabi's Theorem, and using Theorem \ref{thm:conn-comp-vis-graph}, we shall now prove Theorem \ref{thm:quadrics}.

\begin{proof}[Proof of Theorem \ref{thm:quadrics}]
\begin{align}
\lim_{s \to \infty} \frac{\EX{b_0(\Gamma)}}{s} &= \lim_{s \to \infty}\frac{\EX{b_0(\mathcal{H}(N, \Pos_n, s))}}{s} \\
&= \lim_{s \to \infty}\frac{\EX{b_0(\graph(N, \Pos_n, s)) - \sum_{i=1}^s \indicator{q_i \in \Pos_n}}}{s} \\
&= \lim_{s \to \infty} \frac{\EX{b_0(\graph(N, \Pos_n, s))}}{s} - \frac{s\cdot \PR{q \in \Pos_n}}{s} \\
&\le \frac{\vol{\Pos_n}}{\vol{\RP^N}} - \frac{\vol{\Pos_n}}{\vol{\RP^N}} \eqcomment{by Theorem \ref{thm:conn-comp-vis-graph} and \eqref{eqn:prob-quadric-positive}} \\
&= 0 \label{eqn:exp-b0-quadrics}
\end{align}
Equation \eqref{eqn:exp-b0-quadrics} together with the fact that $\lim_{s \to \infty} \frac{\EX{b_0(\Gamma)}}{s}$ is obviously non-negative completes the proof.
\end{proof}

\subsection{Average number of connected components of obstacle random graphs}
\label{sec:avg-conn-comp}

In this section, we shall prove Theorem \ref{thm:conn-comp-vis-graph}. Below is a synopsis of the proof.

\begin{enumerate}
\item An important definition is the definition of the \emph{good cone} of a point. The \emph{good cone} of a point $p$ w.r.t. $\Pos$ is defined as the set of all points in $\RP^N \setminus \Pos$, which if sampled, would be connected by an edge to $p$ (see Definition \ref{defn:good-cone}). Thus, after sampling a point $p$, if we were to sample another point $q$ which happens to be in the good cone of $p$, then $p$ and $q$ would be connected by an edge.
\item The general strategy is to prove that, w.h.p, after a constant number of samples (say $\kappa$), the entire sphere is ``covered''. Here covered is defined as the state where the union of the good cones of the already sampled points contains $\RP^N \setminus \Pos$. By definition, this means that if we were to sample any additional points, they will be connected to at least one of the already sampled $\kappa$ vertices by an edge. Consequently, additional points cannot create a new connected component, which in turn means that w.h.p., there are at most $\kappa$ connected components.
\item To prepare the ground for showing the above, a simple first step is Proposition \ref{prop:good-event-probability}, where we understand the distribution of the number of vertices in various regions in $\RP^N$. Specifically, Proposition \ref{prop:good-event-probability} gives tail bounds on the number of vertices in $\Pos$, $\Pos(\eps) \setminus \Pos$ (where $\Pos(\eps)$ is the $\eps$-neighbourhood of $\Pos$) and $\RP^N \setminus \Pos(\eps)$.
\item Lemma \ref{lem:covering-point} proves that the subgraph of $\graph(N, \Pos, s)$ restricted to $\RP^N \setminus \Pos(\eps)$ has number of connected components constant w.r.t. $s$. The proof of Lemma \ref{lem:covering-point} involves the following sub-steps.
\begin{enumerate}
\item Cover $\RP^N \setminus \Pos(\eps)$ with balls of radius $r > 0$, $r$ to be chosen later.
\item  Then for each $r$-ball $B$, we proceed to lower bound the probability (Lemma \ref{lem:cover-any-ball}) of choosing a point in $\RP^N \setminus \Pos(\eps)$ such that the good cone of the point contains $B$. This involves showing that the volume of the good cone of a point is a continuous function of the position of the point, thus it attains a minimum. We prove this by first considering a smooth approximation of $\Pos$ containing $\Pos$ and contained in $\Pos(\eps)$ (Proposition \ref{prop:smooth-approximation-exists}), and then applying a stereographic projection and proving continuity in Euclidean space (Lemma \ref{lem:positive-measure-euclidean}).
\item Now, we have a finite number of balls covering $\RP^N \setminus \Pos(\eps)$, and there is a lower bound on the probability of each of these balls being covered. To complete our proof, we just need to show that after a finite number of samples, all balls will be covered. The final step, which is similar to a coupon-collector type argument (Lemma \ref{lem:geometric-coupon-collector}), gives tail bounds on the number of points required for all $r$-balls to be contained in good cones. This ensures that any new point sampled in $\RP^N \setminus \Pos(\eps)$ will not add a new connected component to the graph.
\end{enumerate}
\end{enumerate}

Sample $s$ points $q_1, \ldots, q_s$ i.i.d. from the uniform distribution on $\RP^N$. Let $\Pos(\eps)$ be the $\eps$-neighbourhood of $\Pos$ in $\RP^N$. Define the random variables
\be
s_e(\eps) = \sum_{i=1}^s \indicator{q_i \in \RP^N\setminus \Pos(\eps)},
\ee 
which is the number of points in $\RP^N\setminus \Pos(\eps)$, 
\be
s_a(\eps) = \sum_{i=1}^s \indicator{q_i \in \Pos(\eps) \setminus \Pos},
\ee 
which is the number of points in $\Pos(\eps) \setminus \Pos$, and 
\be
s_p = \sum_{i=1}^s \indicator{q_i \in \Pos},
\ee 
which is the number of points in $\Pos$. Obviously, \be
s = s_e(\eps) + s_a(\eps) + s_p.
\ee
Fix any $\beta < \nicefrac{1}{2}$. Now, let $\Omega_1(\eps), \Omega_2(\eps), \Omega_3$ be the following defined events:
\begin{align}
\Omega_1(\eps) &= \left\{s_e(\eps) = s\cdot\left(1 - \frac{\vol{\Pos(\eps)}}{\vol{\RP^N}}\right) \pm s^{\nicefrac{1}{2} + \beta}\cdot\left(1 - \frac{\vol{\Pos(\eps)}}{\vol{\RP^N}}\right)\right\}, \\
\Omega_2(\eps) &= \left\{s_a(\eps) = s\cdot\left(\frac{\vol{\Pos(\eps)\setminus \Pos}}{\vol{\RP^N}}\right) \pm s^{\nicefrac{1}{2} + \beta}\cdot\left(\frac{\vol{\Pos(\eps)\setminus \Pos}}{\vol{\RP^N}}\right)\right\}, \\
\Omega_3 &= \left\{s_p = s\cdot\left(\frac{\vol{\Pos}}{\vol{\RP^N}}\right) \pm s^{\nicefrac{1}{2} + \beta}\cdot\left(\frac{\vol{\Pos}}{\vol{\RP^N}}\right)\right\}.
\end{align}

Below we have a simple proposition that gives tail bounds on the random variables $s_e(\eps)$, $s_a(\eps)$ and $s_p$.

\begin{proposition}
\label{prop:good-event-probability}
Let $\beta$ be as chosen before. For all $0 < \delta < 1$, $\eps > 0$, there exists $\tilde{s}_1 = \tilde{s}_1(\delta, \beta, \eps, N)$, such that if $s > \tilde{s}_1$,
\be
\PR{\Omega_1(\eps)^c },  \PR{\Omega_2(\eps)^c }, \PR{\Omega_3^c }< \frac{\delta}{3},
\ee
and consequently, for all $\eps > 0$,
\be
\PR{\Omega_1(\eps) \cap \Omega_2(\eps) \cap \Omega_3} > 1 - \delta.
\ee
This also implies that for all $\eps > 0$,
\be
\lim_{s \to \infty} \PR{\Omega_1(\eps) \cap \Omega_2(\eps) \cap \Omega_3} = 1.
\ee
\end{proposition}

See Appendix \ref{sec:proof-of-tail-bounds} for the proof of Proposition \ref{prop:good-event-probability}.

Recall that $\graph(N, \Pos, s)$ is the graph over all the $s$ points $q_1, \ldots, q_s$. Let $\graph_1(N, \Pos, s, \eps)$ denote the subgraph of $\graph(N, \Pos, s)$ restricted to the vertices in $\RP^N \setminus \Pos(\eps)$, let $\graph_2(N, \Pos, s, \eps)$ denote the subgraph of $\graph(N, \Pos, s)$ restricted to the vertices in $\Pos(\eps) \setminus \Pos$, and let $\graph_3(N, \Pos, s)$ denote the subgraph of $\graph(N, \Pos, s)$ restricted to the vertices in $\Pos$. Note that $\graph_3(N, \Pos, s)$ contains $s_p$ vertices and no edges whatsoever. The following lemma, whose proof is postponed, gives us some information of the distribution of the zeroth Betti number of $\graph_1(N, \Pos, s, \eps)$.

\begin{lemma}
\label{lem:covering-point}
For all $\eps > 0$, $\delta_1 > 0$, there exists $\tilde{s}_2 = \tilde{s}_2(\eps, \delta_1, N)$, $a = a(\eps, N)$, such that for all $s > \tilde{s}_2$
\be
\PR{b_0(\graph_1(N, \Pos, s, \eps)) \le \frac{\tilde{s}_2}{a} \condition \Omega_1(\eps)} \ge 1 - \delta_1.
\ee
\end{lemma}

\begin{definition}
\label{defn:good-cone}
For any point $q \in \RP^N$, and $\mathcal{F} \subseteq \RP^N$, define the \emph{good cone of $q$ w.r.t. $\mathcal{F}$} as
\be
g_q(\mathcal{F}) = \left\{x \in \RP^N \mid \ell(q, x) \cap \mathcal{F} = \emptyset\right\}.
\ee
If $\mathcal{F}$ is clear from the context, we will refer to $g_q(\mathcal{F})$ as just the \emph{good cone} of $q$.
\end{definition}

Thus $g_q(\Pos)$ is a random variable that denotes the set of points in $\RP^N$ which, if sampled, would be connected to $q$ by an edge in $\graph(N, \Pos, s)$. See Figure~\ref{fig:good-cone} for an example illustration of the good cone. Next, for $B \subseteq \RP^N$, $\eps > 0$, define
\be
G_B(\mathcal{F}) = \{x \in \RP^N \setminus \Pos(\eps) \mid g_x(\mathcal{F}) \supseteq B \cap \left(\RP^N \setminus \Pos(\eps)\right)\}.
\ee
In words, $G_B(\mathcal{F})$ denotes all points $x \in \RP^N \setminus \Pos(\eps)$ such that the good cone of $x$ w.r.t. $\mathcal{F}$ completely contains whatever part of $B$ is outside $\Pos(\eps)$. The following lemma, whose proof is also postponed, gives a lower bound on the relative volume of $g_q(\Pos)$, when $q$ is outside $\Pos(\eps)$.

\begin{illustration}
\begin{figure}
\vspace{5pt}
\begin{tikzpicture}[scale=1.7]
%% Draw Sphere
       \pgfmathsetmacro\R{2} 
	   \fill[ball color=\spherecolor,opacity=\sphereopacity] (0,0) circle (\R); % 3D lighting effect

%% Cone of Positive Quadratic Forms
		\fill[fill=\pncolor,opacity=\pnopacity] {[rotate=0] (48:2 and 0.9) arc (48:100:2 and 0.9)} {[rotate=320] arc (115:73.5:2 and 0.5)} {[rotate=220] arc(106.5:136.5:2 and 0.5)}-- cycle;
        \node[above,color=black]  at (1.0,0.35) {\large $\Pos$};

%% tangents
		\geodesicA{0}{180}{2}{0.3}{95}{draw,dashed,thick,\nonedgecolor}
		\geodesicA{180}{360}{2}{0.8}{49}{draw,dashed,thick,\nonedgecolor}

%% shading good cone
		\fill[fill=\goodconecolor,opacity=\goodconeopacity] {[rotate=95] (180:2) arc (180:126.5:2 and 0.3)} {[rotate=49] arc (240:360:2 and 0.8)} {[rotate=0] arc (49:-85:2)}-- cycle;
		\fill[fill=\goodconecolor,opacity=\goodconeopacity] {[rotate=49] (180:2) arc (180:240:2 and 0.8)} {[rotate=275] arc (306.5:180:2 and 0.3)} {[rotate=0] arc (95:229:2)};
        \node[above,color=black]  at (1.2,-1) {\large $g_q(\Pos)$};
        \node[above,color=black]  at (-1,-0.5) {\large $g_q(\Pos)$};

%% Label sphere
        \node[above,color=black]  at (0,1.6) {\large $\RP^N$};

%% q
		\vertex{0.8}{1.22}{260}{0}
        \node[above,color=black]  at (-0.2,-1.6) {\large $q$};
\end{tikzpicture}
\caption{Illustration of $g_q(\Pos)$, the \emph{good cone} of a point $q$ w.r.t. $\Pos$. The dashed lines are geodesics which are tangent to $\Pos$ and incident on $q$. The shaded region is $g_q(\Pos)$. Recall that in $\graph(N, \Pos, s)$, by definition, if $q$ is sampled and any point in $g_q(\Pos)$ is sampled, these points would be connected to each other by any edge.}
\label{fig:good-cone}
\end{figure}
\end{illustration}

\begin{lemma}
\label{lem:cover-any-ball}

For $p \in \RP^n$, and $r \ge 0$, define $B(p, r)$ to be a ball centered at $p$ of radius $r$. For all $\eps > 0$, there exists $r = r(\eps, N) > 0$, $\delta_2 = \delta_2(\eps, N)$, such that for any $p \in \RP^N \setminus \Pos(\eps)$,
\be
\frac{\vol{G_{B(p, r)}(\Pos)}}{\vol{\RP^N}} \ge \delta_2.
\ee
\end{lemma}

\begin{comment}
Define \be
\alpha^{r}: \RP^N \setminus \mathrm{int}(\Pos(\eps)) \to [0, \infty), \qquad \text{which takes } p \mapsto \frac{\vol{G_{B(p, r)}(\Pos)}}{\vol{\RP^N}}.
\ee
where $r \le \nicefrac{\eps}{8}$ is going to be chosen later. Note that since we are going to be choosing $r \le \nicefrac{\eps}{8}$,
\be 
B(p, r) \subseteq \RP^N \setminus \Pos(\nicefrac{\eps}{2}), \qquad \forall p \in \RP^N \setminus \Pos(\eps).
\ee
\end{comment}

Later, we will choose an $r \le \nicefrac{\eps}{8}$, so that
\be 
B(p, r) \subseteq \RP^N \setminus \Pos(\nicefrac{\eps}{2}), \qquad \forall p \in \RP^N \setminus \Pos(\eps).
\ee

\begin{remark}\label{remark:convex}Observe that the convex set $\Pos \subset \RP^N$ is contained in one single affine chart, and therefore if we denote by $f:S^N\to \RP^N$ the double cover map, the preimage $f^{-1}(\mathcal{P})$ consists of two isometric components, which for simplicity we still denote by $\Pos$ and $-\Pos$. Both these components are entirely contained in a open hemisphere, and we assume that $\Pos$  is contained in \be U=\textrm{int}B\left(e_0, \frac{\pi}{2}\right)\subset S^N\ee
 for some point $e_0\in S^N$. Let us denote now by
 \be \sigma:U\to \R^N\ee
  the stereographic projection constructed as follows: we identify $\R^N$ with $T_{e_0}S^N$ and for every point $y\in U$ we take $\sigma(y)$ to be the point of intersection between $T_{e_0}S^N$ and the line from the origin to $y$.
This stereographic projection has an interesting property that we will use: it maps (unparametrized) geodesics entirely contained in $U$, i.e. intersections between $U$ and great circles, to (unparametrized) geodesics in $\R^N$, i.e. straight segments. In particular $\sigma$ maps convex sets to convex sets, and the same is true for its inverse. In particular we can use results from convex geometry in $\R^N$ to obtain results for the convex geometry of $U$. Since $f|_{U}:U\to \RP^n$ is a local isometry onto its image, the same is true for the geometry of convex sets in $\RP^N$. 
\end{remark}

The next proposition is an application of the idea explained in Remark \ref{remark:convex}.

\begin{proposition}
\label{prop:smooth-approximation-exists}
For all $\eps > 0$, there exists a smooth convex set $\tilde{\Pos}(\eps)$ such that $\Pos \subseteq \tilde{\Pos}(\eps) \subseteq \Pos(\eps)$. 
\end{proposition}

\begin{proof}
Consider the $\nicefrac{\eps}{2}$-neighbourhood of $\Pos$, i.e. $\Pos(\nicefrac{\eps}{2})$. Since the set of smooth convex bodies is dense in the Hausdorff distance induced topology on the space of convex bodies (see \cite[Theorem 2.7.1.]{schneider2014convex}), there exists a body $C_{\eps}$ that is convex, smooth and also satisfies \be
\label{eqn:dist-approximation-from-center}
d_H(C_{\eps}, \Pos(\nicefrac{\eps}{2})) \le \nicefrac{\eps}{3},
\ee
where $d_H$ denotes Hausdorff distance with the underlying metric being the usual round metric on $S^N$. We shall now show that $C_{\eps}$ itself is the smooth approximation we desire, i.e. $\tilde{\Pos}(\eps)$. We know that $d_H(\Pos, \Pos(\nicefrac{\eps}{2})) = \nicefrac{\eps}{2}$. Observe that if $\Pos$ was not completely contained in $C_{\eps}$, then $d_H(\Pos(\nicefrac{\eps}{2}), C_{\eps}) \ge \nicefrac{\eps}{2}$, which contradicts Equation \eqref{eqn:dist-approximation-from-center}. Similarly, it can be shown that $C_{\eps}$ is completely contained in $\Pos(\eps)$ because otherwise, we would again have $d_H(\Pos(\nicefrac{\eps}{2}), C_{\eps}) \ge \nicefrac{\eps}{2}$ (because $d_H(\Pos(\nicefrac{\eps}{2}), \Pos(\eps)) = \nicefrac{\eps}{2}$) contradicting Equation \eqref{eqn:dist-approximation-from-center}.
\end{proof}

The following lemma proves that for every $r'$-ball (where $r' > 0$ is appropriately chosen) contained in $\RP^N \setminus \Pos(\eps)$, there is a set of positive measure such that the good cone of any point in this set contains the ball, which in turn implies that with each vertex sampled, there is a positive probability that a particular $r'$ ball is \emph{covered}.

\begin{lemma}
\label{lem:positive-measure-euclidean}
For all $\eps > 0$, there exists $r' = r'(\eps, N), \delta_2' = \delta_2'(\eps, N) > 0$ such that for any $p \in \RP^N \setminus \Pos(\eps)$,
\be
\frac{\vol{\sigma\left(G_{B(p, r')}(\Pos)\right)}}{\vol{\sigma\left(S^N\right)}} \ge \delta_2'.
\ee
\end{lemma}
\begin{proof}
 Let $\QQ_n(\eps) = \sigma(\Pos(\eps)) \subseteq \R^N$, and $\tilde{\QQ}(\eps) = \sigma(\tilde{\Pos}(\eps))$ (cf. Proposition \ref{prop:smooth-approximation-exists}). Note that for any $p \in \RP^N$, $g_p(\Pos) \supseteq g_p(\tilde{\Pos}(\eps))$, and for any $B \subseteq \RP^N$, $G_B(\Pos) \supseteq G_B(\tilde{\Pos}(\eps))$ (see Figure~\ref{fig:good-cone-for-approximation} for an illustration).

\begin{illustration}
\begin{figure}

\vspace{5pt}
\begin{tikzpicture}[scale=1.7]
%% Draw Sphere
    \pgfmathsetmacro\R{2} 
	\fill[ball color=\spherecolor,opacity=\sphereopacity] (0,0) circle (\R); % 3D lighting effect

%% Cone of positive quadratic forms
	\fill[fill=\pncolor,opacity=\pnopacity] {[rotate=0] (48:2 and 0.9) arc (48:100:2 and 0.9)} {[rotate=320] arc (115:73.5:2 and 0.5)} {[rotate=220] arc(106.5:136.5:2 and 0.5)}-- cycle;
	\fill[fill=\pncolor,opacity=\pntildeopacity,rounded corners=6.2pt] {[rotate=0] (38:2 and 1.1) arc (38:114.7:2 and 1.1)}  {[rotate=320] arc (129.7:68.5:2 and 0.3)} {[rotate=220] arc(106.5:136.5:2 and 0.5)}-- cycle; %smooth approximation

%% shading good cone
		\fill[fill=\goodconecolor,opacity=\goodconeopacity] {[rotate=95] (180:2) arc (180:126.5:2 and 0.3)} {[rotate=49] arc (240:360:2 and 0.8)} {[rotate=0] arc (49:-85:2)}-- cycle;
		\fill[fill=\goodconecolor,opacity=\goodconeopacity] {[rotate=49] (180:2) arc (180:240:2 and 0.8)} {[rotate=275] arc (306.5:180:2 and 0.3)} {[rotate=0] arc (95:229:2)};
		\fill[fill=\smallergoodconecolor,opacity=0.15] {[rotate=106.5] (180:2) arc (180:124:2 and 0.57)} {[rotate=42] arc (243:360:2 and 0.89)} {[rotate=0] arc (42:-73:2)}-- cycle;
		\fill[fill=\smallergoodconecolor,opacity=0.15] {[rotate=42] (180:2) arc (180:243:2 and 0.89)} {[rotate=106.5] arc (124:0:2 and 0.57)} {[rotate=0] arc (106.5:222:2)}-- cycle;

        \node[above,color=black]  at (1.2,-1) {\large $g_q(\tilde{\Pos}(\eps))$};
        \node[above,color=black]  at (-0.95,-0.8) {\large $g_q(\tilde{\Pos}(\eps))$};

	    \node[above,color=black]  at (1.0,0.35) {\large $\mathcal{P}$};
    	\node[above,color=black]  at (0,0.83) {\small $\tilde{\mathcal{P}}(\eps)$};

%% tangents
	\geodesicA{0}{180}{2}{0.3}{95}{draw,dashed,thick,\nonedgecolor}
	\geodesicA{180}{360}{2}{0.8}{49}{draw,dashed,thick,\nonedgecolor}
	\geodesicA{0}{180}{2}{0.57}{106.5}{draw,thick,dotted,\smallergoodconecolor}
	\geodesicA{180}{360}{2}{0.89}{42}{draw,thick,dotted,\smallergoodconecolor}

%% Label sphere
    \node[above,color=black]  at (0,1.6) {\large $\RP^N$};

%% q
    \vertex{0.8}{1.22}{260}{0}
    \node[above,color=black]  at (-0.2,-1.6) {\large $q$};
\end{tikzpicture}\caption{Illustration of the good cone of $q$ w.r.t. $\tilde{\Pos}(\eps)$. $\tilde{\Pos}(\eps)$ is an approximation of $\Pos$ which is convex and has a smooth boundary, such that $\Pos \subseteq \tilde{\Pos}(\eps) \subseteq \Pos(\eps)$. The dashed lines are geodesics which are tangent to $\Pos$ and incident on $q$, and the dotted lines are geodesics which are tangent to $\tilde{\Pos}(\eps)$ and incident on $q$. Observe that $g_q(\tilde{\Pos}(\eps)) \subseteq g_q(\Pos)$, and consequently, $\vol{g_q(\tilde{\Pos}(\eps))} \le \vol{g_q(\Pos)}$.}
\label{fig:good-cone-for-approximation}
\end{figure}
\end{illustration}

Define the map
\be
\tilde{\alpha}^{s}: S^{N-1} \setminus \mathrm{int}(\QQ(\eps)) \to [0, \infty), \qquad \text{which takes } q \mapsto \frac{\vol{\sigma\left(G_{B(\sigma^{-1}(q), s)}(\tilde{\Pos}(\eps))}\right)}{\vol{S^{N-1}}}.
\ee
To establish the lemma, we need to show that for an appropriately chosen $r$, $\tilde{\alpha}^r$ attains a minimum on its domain. As a first step, we shall show that the map
\be
\tilde{\alpha}^{0}: S^{N-1} \setminus \mathrm{int}(\QQ(\eps)) \to [0, \infty), \qquad \text{which takes } q \mapsto \frac{\vol{\sigma\left(g_{\sigma^{-1}(q)}(\tilde{\Pos}(\eps))}\right)}{\vol{S^{N-1}}},
\ee
is bounded below by a continuous function.

\begin{figure}
  \begin{tikzpicture}[scale=0.5]
	% Convex body
    \draw[decoration={closed contour},decorate] plot[smooth cycle,tension=0.5] coordinates {(0,0) (2,0) (3,1) (0,2)};
	\node[above]  at (0.7,0.35) {\large $\tilde{\QQ}_n(\eps)$};

	% Tangent plane
	\draw[] (3.8,-2.5) -- (3.8,4.5);
	\node[below] at (3.8, -2.5) {\large $T_{q'}(\partial \tilde{\QQ}_n(\eps))$};

	% S^(N-1)
	\draw (3.2,1) circle (5.5);
	\node[above] at (2.5, 5.3) {\large $S^{N-1}$};

	% p
	\fill[] (6.5,1) circle (3pt);
	\node[right] at (6.5,1) {\large $q$};

	%p'	
	\fill[] (3.8,1) circle (3pt);
	\node[left] at (3.8,1) {\large $q'$};

	%projection of convex body
	\draw[dashed,{Circle[black]}->] (0,2.75) -- (3.8,2.75);
	\draw[dashed,{Circle[black]}->] (0.5,-0.9) -- (3.8,-0.9);
	\draw[very thick] (3.8, 2.75) -- (3.8,-0.9);
    \path[-latex] (6, -1) edge [bend right] (3.8,0);
	\node[below] at (6, -1) {\large $\Pi(\tilde{\QQ}_n(\eps))$};

	%angle at p
	\draw (6.5,1) -- (-0.17,5.33);
	\draw (6.5,1) -- (3.8,1);
	\draw (5.5,1) arc (180:115:0.5);
	\node[left] at (5.7,1.4) {\footnotesize $\beta(q)$};
  \end{tikzpicture}
\end{figure}

Let $q'$ be the point shortest to $q$ on $\partial \tilde{\QQ}_n(\eps)$, the boundary of $\tilde{\QQ}_n(\eps)$. Let $\Pi(\tilde{\QQ}_n(\eps))$ be the projection of $\tilde{\QQ}_n(\eps)$ onto $T_{q'}(\partial \tilde{\QQ}_n(\eps))$, the tangent space of $\tilde{\QQ}_n(\eps)$ at $q'$. Observe that
\be
\lambda(q) = \max_{v \in \Pi(\tilde{\QQ}_n)(\eps)} \|v\|_2
\ee
is continuous. Consequently, observe that
\be
\frac{\vol{\sigma\left(g_{\sigma^{-1}(q)}(\tilde{\Pos}_n(\eps))\right)}}{\vol{S^{N-1}}} \ge 1 - \underbrace{\frac{\vol{\text{spherical cap with angle }\tan^{-1}\left(\frac{\lambda(q)}{2\|q - q'\|_2}\right)}}{\vol{S^{N-1}}}}_{\beta(q)}.
\ee
$\beta(q)$ is a continuous function, and thus attains a maximum on $S^{N-1} \setminus \mathrm{int}(\QQ_n(\eps))$ (remember that $\|q - q'\|_2$ can never become $0$ because $q$ is always outside $\Pos(\eps)$) proving that $\tilde{\alpha}^0$ is bounded below by a continuous function that attains a minimum on its domain. 

From this, we have that for every $p \in \RP^N \setminus \Pos(\eps)$, we can find a direction $\vec{v}'$ in $\R^N$ and an angle $\theta$ such that for all directions $\vec{v}$ with $\cos^{-1} \frac{\vec{v} \cdot \vec{v}'}{\|\vec{v}\|_2\|\vec{v}'\|_2} \le \theta$, we have that 
$\ell_{\sigma}(p, \vec{v}) \subseteq \sigma\left(g_p(\tilde{\Pos}(\eps))\right)$, where $\ell_{\sigma}(p, \vec{v})$ denotes the line in $\R^N$ through $\sigma(p)$ in the direction $\vec{v}$. Note that $\vec{v}'$ and $\theta$ depend on $p$ continuously. Let $p_{v'}$ be the point of intersection of the line $\ell_{\sigma}(p, \vec{v}')$ and $S^{N-1}$, and now let $p_2$ be the mid-point on the line joining $p$ and $p_{v'}$. Since $\theta$ depends on $p$ continuously, it has a minimum on $\RP^N \setminus \mathrm{int}(\Pos(\eps))$, and thus we can pick $r'' = r''(\eps, N) > 0$ such that $B(p_2, r'') \subseteq \sigma\left(g_p(\tilde{\Pos}(\eps))\right)$.

Now, for the sake of contradiction, assume that for all $r' > 0$, $\min_{q \in S^{N-1}\setminus \mathrm{int}(\QQ_n(\eps))} \tilde{\alpha}^{r'}(q) = 0$, and let $q$ be the point at which $\tilde{\alpha}^{r'}$ attains the minimum. Then we can find a sequence $(r_n)$, with $r_n \to 0$, and a sequence $(q_n)$, with $q_n \to q$, where $q_n \in B(p, r_n)$, such that for all $n < \infty$, there exists a point $b_n \in B(p_2, r'')$ with $b_n \not\in g_{q_n}$. Since $S^{N-1} \setminus \mathrm{int}({\QQ(\eps)})$ is compact, and $B(p_2, r'')$ is obviously compact as well, this means that $(\lim_{n \to \infty} b_n) \not\in (\lim_{n \to \infty} g_{q_n})$, implying that there is a point in $B(p_2, r'')$ which does not belong to $g_q$, which gives us the contradiction we require.
\end{proof}

\begin{proof}[Proof of Lemma \ref{lem:cover-any-ball}]
Set $r = \min(r', \nicefrac{\eps}{8})$. The proof of the lemma follows by noting that since $\sigma$ is smooth, bijective and angle-preserving (conformal)\footnote{Note that the stereographic projection is not isometric, and thus does not preserve areas. However, angle-preservation is enough for us.}, proving that there is a set of strictly positive measure that is good for all $r$-balls centered in $\RP^N\setminus\Pos(\eps)$ follows from Lemma~\ref{lem:positive-measure-euclidean}. This is because since $\delta_2' > 0$, the pre-image under $\sigma$ of any set of measure at least $\delta_2'$ will be strictly positive ($\delta_2$ will be the measure of the pre-image, under $\sigma$, of the set in $\R^N$ which attains the minimum measure $\delta_2'$).
\end{proof}

The lemma below gives bounds on the number of samples from $\RP^N \setminus \Pos(\eps)$ required to cover all of $\RP^N \setminus \Pos(\eps)$ with good cones.

\begin{lemma} 
\label{lem:geometric-coupon-collector}
Let $q'_i$ denote the $i^\text{th}$ sampled point. For any $\eps > 0$, define $C = C(\eps)$ to be a random variable that denotes the minimum index such that the union of the good cones of $q'_1 \ldots q'_C$ covers $\RP^N \setminus \Pos(\eps)$, i.e.
\be
\bigcup_{i=1}^C g_{q'_i}(\Pos_n) \supseteq \RP^N \setminus \Pos(\eps).
\ee
Then, for all $\delta_3 > 0$, there exists $\gamma = \gamma(\eps, \delta_3, N)$ such that
\be
\PR{C \le \gamma} \ge 1 - \delta_3.
\ee
\end{lemma}

\begin{proof}
Take a covering of $\RP^N \setminus \Pos(\eps)$ with $r$-balls (where $r$ is from Lemma \ref{lem:cover-any-ball}) of size $Q = Q(\eps, N)$, and let the $Q$ balls that cover $\RP^N \setminus \Pos(\eps)$ be $B_1, \ldots, B_Q$. Remember that the conditional distribution of sampling from $\RP^N \setminus \Pos(\eps)$ is uniform. Let $C_i$ denote the additional number of points needed to be sampled from $\RP^N \setminus \Pos(\eps)$ such that $B_i$ is covered, given that balls $B_1, \ldots B_{i-1}$ are already covered by $\bigcup_{i=1}^{C_{i-1}} g_{q'_i}$. By definition,
\be
C \le \sum_{i=1}^Q C_i.
\ee
When balls $B_1, \ldots B_{i-1}$ are already covered, $B_i$ could already be covered. Let the probability that $B_i$ is already covered be $p_i$. If not covered, by Lemma~\ref{lem:cover-any-ball}, each $C_i$ is a geometric random variable with parameter $\mu_i \ge \delta_2$. This means
  \be
    C_i = \begin{cases} 0 &\mbox{with prob. } p_i \\ 
                                  \mathrm{\text{Geom}}(\mu_i) & \mbox{with prob. } 1 - p_i
              \end{cases},
  \ee
where $\text{Geom}(\mu)$ refers to the geometric distribution with mean $\frac{1}{\mu}$. Thus
\be
\EX{C_i} = 0\cdot p_i + (1-p_i) \cdot \frac{1}{\mu_i} \le \frac{1}{\mu_i} \le \frac{1}{\delta_2},
\ee
and by linearity of expectation, in turn, we get that
\be
\label{eqn:expected-number-to-cover}
\EX{C} \le \frac{Q}{\delta_2}.
\ee
Set $\gamma = \frac{Q}{\delta_2\delta_3}$. Applying Markov's inequality on $C$, and using Equation~\eqref{eqn:expected-number-to-cover}, the lemma follows.
\end{proof}

\begin{proof}[Proof of Lemma \ref{lem:covering-point}]
Lemma \ref{lem:geometric-coupon-collector} shows that we will have a covering of $\RP^N \setminus \Pos(\eps)$ with \emph{good sets}, with probability at least $1 - \delta_3$, if we have $s_e(\eps) \ge \gamma$. To complete the proof of Lemma~\ref{lem:covering-point}, we have to set $\tilde{s}_2$ appropriately so that if $s \ge \tilde{s}_2$, then $s_e(\eps) \ge \gamma$. Conditioning on $\Omega_1(\eps)$, it is clear that if $s \ge k\cdot\gamma\left(\frac{\vol{\RP^N}}{\vol{\RP^N} - \vol{\Pos(\eps)}}\right)$, for an appropriately chosen constant $k$, then $s_e(\eps) \ge \gamma$. Thus, conditioned on $\Omega_1(\eps)$, setting $\tilde{s}_2 = k\cdot\gamma\left(\frac{\vol{\RP^N}}{\vol{\RP^N} - \vol{\Pos(\eps)}}\right)$ ensures we have a covering of $\RP^N \setminus \Pos(\eps)$ with \emph{good sets} with probability at least $1 - \delta_3$. 

Since, $\RP^N \setminus \Pos(\eps)$ is covered, any new point that is added to $\RP^N \setminus \Pos(\eps)$ will be connected to at least one of the existing $\gamma$ vertices, which in turn means that the number of connected components of the graph stays fixed as $\gamma$. The lemma follows by setting $a = k\left(\frac{\vol{\RP^N}}{\vol{\RP^N} - \vol{\Pos(\eps)}}\right)$.
\end{proof}

\begin{proof}[Proof of Theorem~\ref{thm:conn-comp-vis-graph}]
We shall prove that, for all $\eps, \delta, \delta_1$, $\lim_{s \to \infty} \frac{\EX{b_0(\graph(N, \Pos, s))}}{s}$ is bounded from above by $\frac{\vol{\Pos_n}}{\vol{\RP^N}}$ plus some terms which depend on $\eps, \delta, \delta_1$. We know that the number of connected components of a graph is bounded from above by the sum of the number of connected components of subgraphs of the graph that form a decomposition of the original graph. Consider the decomposition of the graph from earlier, i.e., $\graph(N, \Pos, s)$ is decomposed into $\graph_1(N, \Pos, s, \eps)$, which is the subgraph of $\graph(N, \Pos, s)$ restricted to the vertices in $\RP^N \setminus \Pos(\eps)$, $\graph_2(N, \Pos, s, \eps)$, which is the subgraph of $\graph(N, \Pos, s)$ restricted to the vertices in $\Pos(\eps) \setminus \Pos$, and $\graph_3(N, \Pos, s)$, which is the subgraph of $\graph(N, \Pos, s)$ restricted to the vertices in $\Pos$. Thus,
\be
\label{eqn:G-decompose}
\EX{b_0(\graph(N, \Pos, s))} \le \EX{b_0(\graph_1(N, \Pos, s, \eps))} + \EX{b_0(\graph_2(N, \Pos, s, \eps))} + \EX{b_0(\graph_3(N, \Pos, s))}.
\ee

To bound $\EX{b_0(\graph_1(N, \Pos, s, \eps))}$, we break the sample space into three disjoint evens, i.e., $\Omega_1(\eps) \cap \left(b_0(\graph_1(N, \Pos, s, \eps)) \le \nicefrac{\tilde{s}_2}{a}\right)$, $\Omega_1(\eps) \cap \left(b_0(\graph_1(N, \Pos, s, \eps)) \le \nicefrac{\tilde{s}_2}{a}\right)^c$, and $\Omega_1(\eps)^c$, where $\tilde{s}_2$ and $a$ are from Lemma \ref{lem:covering-point}. Thus, for any $\eps > 0$, we have
\begin{align}
\EX{b_0(\graph_1(N, \Pos, s))} &\le \underbrace{\int_{\Omega_1(\eps) \cap \left(b_0(\graph_1(N, \Pos, s, \eps)) \le \nicefrac{\tilde{s}_2}{a}\right)} b_0(\graph_1(N, \Pos, s, \eps)) \de{\omega}}_{A} + \underbrace{\int_{\Omega_1(\eps)^c} b_0(\graph_1(N, \Pos, s, \eps)) \de{\omega}}_{B} \\ &\qquad+ \underbrace{\int_{\Omega_1(\eps) \cap \left(b_0(\graph_1(N, \Pos, s, \eps)) \le \nicefrac{\tilde{s}_2}{a}\right)^c} b_0(\graph_1(N, \Pos, s, \eps)) \de{\omega}}_{C}. \label{eqn:G1}
\end{align}
To bound $\EX{b_0(\graph_2(N, \Pos, s, \eps))}$, we break the sample space into two disjoint events, i.e. $\Omega_2(\eps)$ and $\Omega_2(\eps)^c$. Thus, for any $\eps > 0$, we have
\be
\EX{b_0(\graph_2(N, \Pos, s))} \le \underbrace{\int_{\Omega_2(\eps)} b_0(\graph_2(N, \Pos, s, \eps))\de{\omega}}_{D} 
+ \underbrace{\int_{\Omega_2(\eps)^c} b_0(\graph_2(N, \Pos, s, \eps))\de{\omega}}_{E}. \label{eqn:G2}
\ee
Finally, to bound $\EX{b_0(\graph_3(N, \Pos, s))}$, we break the sample space into two disjoint events, i.e. $\Omega_3$ and $\Omega_3^c$. Thus, we have
\be
\EX{b_0(\graph_3(N, \Pos, s))} \le \underbrace{\int_{\Omega_3} b_0(\graph_3(N, \Pos, s))\de{\omega}}_{F} + \underbrace{\int_{\Omega_3^c} b_0(\graph_3(N, \Pos, s))\de{\omega}}_{G}. \label{eqn:G3}
\ee
Thus, from \eqref{eqn:G1}, \eqref{eqn:G2}, \eqref{eqn:G3}, and \eqref{eqn:G-decompose}, we have that
\be
\label{eqn:ub-expected-connected-component}
\EX{b_0(\graph(N, \Pos, s))} \le A + B + C + D + E + F + G.
\ee

Let us now estimate $A, B, C, D, E, F, G$. Because we are integrating over the space where $b_0(\graph_1(N, \Pos, s, \eps)) \le \nicefrac{\tilde{s}_2}{a}$, for all $\eps > 0$, obviously,
\be
\label{eqn:A}
A \le \frac{\tilde{s}_2}{a}.
\ee
We apply the trivial bound of $s$ on $b_0(\graph_1(N, \Pos, s, \eps))$ to get that, for all $\delta > 0$, $\eps > 0$,
\be
\label{eqn:B}
B \le \PR{\Omega_1(\eps)^c}s \le \frac{\delta}{3}s,
\ee
as long as $s \ge \tilde{s}_1 = \tilde{s}_1(\delta, \beta, \eps, N)$ (cf. Proposition~\ref{prop:good-event-probability}). By Lemma~\ref{lem:covering-point}, for all $\delta_1 > 0$, $\eps > 0$, if $s > \tilde{s}_2 = \tilde{s}_2(\eps, \delta_1, N)$, $\PR{b_0(\graph_1(N, \Pos, s, \eps)) > \nicefrac{\tilde{s}_2}{a} \condition \Omega_1(\eps)} < \delta_1$, for some specific $a = a(\eps, N)$. Thus, for $\delta_1 > 0$, $\eps > 0$,
\be
\label{eqn:C}
C \le \PR{\Omega_1(\eps)}\cdot\PR{b_0(\graph_1(N, \Pos, s, \eps)) > \nicefrac{\tilde{s}_2}{a} \condition \Omega_1(\eps)}s \le \PR{\Omega_1(\eps)}\delta_1 s \le \delta_1 s.
\ee
Trivially, $b_0$ of a graph is bounded from above by the number of vertices in the graph. Thus, for all $\eps > 0$,
\begin{align}
D &\le \PR{\Omega_2(\eps)}\left(s\cdot\left(\frac{\vol{\Pos(\eps)\setminus \Pos}}{\vol{\RP^N}}\right) + s^{\nicefrac{1}{2} + \beta}\cdot\left(\frac{\vol{\Pos(\eps)\setminus \Pos}}{\vol{\RP^N}}\right)\right) \\
&\le \left(s + s^{\nicefrac{1}{2} + \beta}\right)\cdot\left(\frac{\vol{\Pos(\eps)\setminus \Pos}}{\vol{\RP^N}}\right). \label{eqn:D}
\end{align}
At the same time, similar to \eqref{eqn:B}, for all $\delta > 0$, $\eps > 0$
\be 
\label{eqn:E}
E \le s. \PR{\Omega_2(\eps)^c} \le s \frac{\delta}{3},
\ee
if $s \ge \tilde{s}_1 = \tilde{s}_1(\delta, \beta, \eps, N)$ (by Proposition~\ref{prop:good-event-probability}). By Equation~\eqref{eqn:relative-vol-postive-definite-cone}, we have that
\be
\label{eqn:F}
F \le \PR{\Omega_3}\left(s\cdot\left(\frac{\vol{\Pos}}{\vol{\RP^N}}\right) + s^{\nicefrac{1}{2} + \beta}\cdot\left(\frac{\vol{\Pos}}{\vol{\RP^N}}\right)\right) \le (s + s^{\nicefrac{1}{2} + \beta})\cdot\left(\frac{\vol{\Pos}}{\vol{\RP^N}}\right).
\ee
Finally, again, for all $\delta > 0$, $\eps > 0$, if $s \ge \tilde{s}_1 = \tilde{s}_1(\delta, \beta, \eps, N)$,
\be
\label{eqn:G}
G \le s. \PR{\Omega_3^c} \le s \frac{\delta}{3}.
\ee
Putting \eqref{eqn:A}, \eqref{eqn:B}, \eqref{eqn:C}, \eqref{eqn:D}, \eqref{eqn:E}, \eqref{eqn:F}, and \eqref{eqn:G} in \eqref{eqn:ub-expected-connected-component}, we have that for all $\eps > 0$, $\delta > 0$, $\delta_1 > 0$,
\be
\label{eqn:expected-connected-component-final}
\lim_{s \to \infty} \frac{\EX{b_0(\graph(N, \Pos, s))}}{s} \le \underbrace{0}_{\nicefrac{A}{s}} + \underbrace{\frac{\delta}{3}}_{\nicefrac{B}{s}} + \underbrace{\delta_1}_{\nicefrac{C}{s}} + \underbrace{\left(\frac{\vol{\Pos(\eps)\setminus \Pos}}{\vol{\RP^N}}\right)}_{\nicefrac{D}{s}} + \underbrace{\frac{\delta}{3}}_{\nicefrac{E}{s}} + \underbrace{\frac{\vol{\Pos}}{\vol{\RP^N}}}_{\nicefrac{F}{s}} + \underbrace{\frac{\delta}{3}}_{\nicefrac{G}{s}}.
\ee
Since Equation \eqref{eqn:expected-connected-component-final} is true for any choice of $\eps, \delta, \delta_1$, we have that
\be
\lim_{s \to \infty} \frac{\EX{b_0(\graph(N, \Pos, s))}}{s} \le \frac{\vol{\Pos}}{\vol{\RP^N}}.
\ee
\end{proof}

%%\subsection{A Ramsey-type result}
\subsection{On the probability of having large cliques in the complement of the quadrics graph $\Gamma$}
%%sb changes
\hide{
The quadrics graph, i.e. $\Gamma$, is a subgraph of $\graph(N, \Pos_n, s)$. 
It is formed by discarding the vertices that fall inside $\Pos_n$ (because the zero sets of quadrics inside $\Pos_n$ are empty). In $\Gamma$, an edge is placed between vertices if the corresponding quadrics intersect. 
Quadrics graphs are random  objects  but for 
for any fixed value of $n$, 
those that occur with positive  probability are  examples of semi-algebraic graphs which we define below.
}
%%sb

%%tbc
We first recall the definition (Definition~\ref{defn:quadrics-graph}) of 
the random quadrics graph $\mathcal{H}(N, \Pos_n, s)$  which is a subgraph of $\graph(N, \Pos_n, s)$. 
Quadrics graphs are random  objects  but for any fixed value of $n$, 
those that occur with positive  probability are  examples of semi-algebraic graphs which we define below.

\begin{definition}
Let $R \subset \mathbb{R}^d \times \mathbb{R}^d$ be a symmetric semi-algebraic relation i.e.  $R$ is a semi-algebraic subset of $\mathbb{R}^d \times \mathbb{R}^d$, and $(x,y) \in R  \Leftrightarrow (y,x) \in R$ for every $x,y \in \mathbb{R}^d$. We say a (undirected) graph $G = (V,E)$ is a $R$-semi-algebraic graph, if there exists a  map $\phi:V \rightarrow \mathbb{R}^d$,  such that for any $u,v \in V,   (u,v) \in E  \Leftrightarrow (\phi(u),\phi(v)) \in R$.   
\end{definition}

The property of existence of large cliques in  semi-algebraic graphs or their complements have been studied
by various authors.
For example,
Alon et al. \cite{Alon-Pach-et-al} prove the following Ramsey-theoretic  theorem.

\begin{theorem}[Alon et al. \cite{Alon-Pach-et-al}]
\label{thm:alon-semialgebraic-graph}
For any symmetric semi-algebraic relation $R$ the following is true. There exists a constant $\delta = \delta(R) > 0$, such that for any $R$-semi-algebraic graph with $n$ vertices one of the following is true.
\begin{enumerate}
\item There exists a clique of size $n^{\delta}$ in $G$.
\item The complement of $G$ has a clique of size $n^{\delta}$. \label{cond:complement-clique}
\end{enumerate}
\end{theorem}

It is clear from its definition that for fixed $n$  each quadrics graph 
%%$\Gamma$, 
$\mathcal{H}(N, \Pos_n, s)$ 
is an $R$-semi-algebraic graph for a fixed  appropriately defined symmetric semi-algebraic relation $R$.
In view of  Theorem~\ref{thm:alon-semialgebraic-graph} it is natural to ask the probability of a quadric graph or its complement  to have a large clique.

The following result rules out with probability $1$,    in the limit as $s \rightarrow \infty$, the existence of large cliques in the complement graph of $\Gamma$ (where large means of size at least a constant fraction of $s$).

\begin{corollary}[of Theorem \ref{thm:quadrics}]
\label{cor:quadrics-graph-ramsey}
Let $\mathcal{H}(N, \Pos_n, s)$ be the random quadrics graph as per Definition \ref{defn:quadrics-graph}. Let $\mathcal{H}(N, \Pos_n, s)^c$ be the complement of the graph on the same set of vertices. Then, for any $\eps > 0$,
\be
\lim_{s \to \infty} \PR{\mathcal{H}(N, \Pos_n, s)^c \text{ contains a clique of size }\eps s} = 0.
\ee
\end{corollary}
\begin{proof}
Let $\Omega_a$ denote the event that there exists a clique of size $\eps s$ in 
%%sb
%%$\Gamma^c$. 
$\mathcal{H}(N, \Pos_n, s)^c$. 
Thus we have
\begin{align}
 0 &= \lim_{s \to \infty} \frac{\EX{b_0(
 %%\Gamma
 \mathcal{H}(N, \Pos_n, s)
 )}}{s}  \eqcomment{by Theorem \ref{thm:quadrics}}\\
&= \lim_{s \to \infty} \frac{\int_{\Omega_a}b_0(
%%\Gamma
\mathcal{H}(N, \Pos_n, s)
) \de{\omega} + \int_{\Omega_a^c}b_0(\Gamma) \de{\omega}}{s} \\
&\ge \lim_{s \to \infty} \frac{\eps s \cdot \PR{\Omega_a} + 0}{s}. \label{eqn:clique-prob-ub}
\end{align}
The last inequality follows by noting that if the complement of 
%%$\Gamma$ 
$\mathcal{H}(N, \Pos_n, s)$ 
contains a clique of size $\eps s$, it means that all $\eps s$ vertices were isolated in 
%%$\Gamma$, 
$\mathcal{H}(N, \Pos_n, s)$,
in turn implying that 
%%$\Gamma$
$\mathcal{H}(N, \Pos_n, s)$ 
 has at least $\eps s$ connected components. By noting that $\lim_{s \to \infty} \PR{\Omega_a}$ is obviously non-negative, and from \eqref{eqn:clique-prob-ub}, we deduce that
\be
\lim_{s \to \infty} \PR{\Omega_a} = 0,
\ee
or in other words
\be
\lim_{s \to \infty} \PR{
%%\Gamma^c 
\mathcal{H}(N, \Pos_n, s)^c
\text{ contains a clique of size }\eps s} = 0.
\ee
\end{proof}

Juxtaposing with Theorem \ref{thm:alon-semialgebraic-graph}, Corollary \ref{cor:quadrics-graph-ramsey} shows that in the quadrics random graph the probability of having an independent set of size at least a fixed fraction of the number of vertices  goes to $0$ with 
%%$n$
$s$.   
This is to be compared with condition (\ref{cond:complement-clique}) in Theorem \ref{thm:alon-semialgebraic-graph},    which posits the existence of an independent set of size at least some fixed polynomial in the number of vertices.  

\bibliographystyle{amsplain}
\bibliography{arrangements}

\appendix
\section{Proof of Proposition \ref{prop:good-event-probability}}
\label{sec:proof-of-tail-bounds}

We will need the additive Chernoff-Hoeffding bound for binomial random variables.

\begin{proposition}[See for e.g. Boucheron et al.~\cite{boucheron2013concentration}]
\label{prop:binomial-concentration}
Let random variable $X \sim \text{Binomial}(n, p)$, and $t \in [0,1]$. Then
\be
\PR{\left|X - \EX{X}\right| > t\cdot\EX{X}} < 2e^{-\nicefrac{t^2\EX{X}}{3}}.
\ee
Consequently, if $n > \tilde{n} = \tilde{n}(t, \delta, p) = \frac{3\log{\nicefrac{2}{\delta}}}{pt^2}$,
\be
\PR{\left|X - \EX{X}\right| > t\cdot\EX{X}} < \delta.
\ee
This also implies that
\be
\label{eqn:concentration-limit}
\lim_{n \to \infty} \PR{\left|X - \EX{X}\right| > t\cdot\EX{X}} = 0.
\ee
\end{proposition}

\begin{proof}[Proof of Proposition \ref{prop:good-event-probability}]
Obviously $s_e(\eps)$, $s_a(\eps)$ and $s_p$ are Binomial random variables. Note that
\be
\EX{s_e(\eps)} = s\cdot\left(1 - \frac{\vol{\Pos(\eps)}}{\vol{\RP^N}}\right).
\ee
Putting $t = s^{\beta - \nicefrac{1}{2}}$ and substituting $\delta$ with $\nicefrac{\delta}{3}$ in Proposition~\ref{prop:binomial-concentration}, we have that if $s > \left(\frac{3\log{\nicefrac{6}{\delta}}}{1 - \frac{\vol{\Pos(\eps)}}{\vol{\RP^N}}}\right)^{\nicefrac{1}{2\beta}}$,
\be
\label{eqn:se-tail}
\PR{\Omega_1(\eps)^c} \le \PR{\left|s_e(\eps) - s\cdot\left(1 - \frac{\vol{\Pos(\eps)}}{\vol{\RP^N}}\right)\right| > s^{\nicefrac{1}{2} + \beta}\cdot\left(1 - \frac{\vol{\Pos(\eps)}}{\vol{\RP^N}}\right)} < \frac{\delta}{3}.
\ee
Similarly, by noting that
\be
\EX{s_a(\eps)} = s \cdot\left(\frac{\vol{\Pos(\eps) \setminus \Pos}}{\vol{\RP^N}}\right),
\ee
and
\be
\EX{s_p} = s \cdot \left(\frac{\vol{\Pos}}{\vol{\RP^N}}\right),
\ee
again by Proposition~\ref{prop:binomial-concentration}, we have that, if $s > \left(\frac{3\log{\nicefrac{6}{\delta}}}{\frac{\vol{\Pos(\eps) \setminus \Pos}}{\vol{\RP^N}}}\right)^{\nicefrac{1}{2\beta}}$,
\be
\label{eqn:sa-tail}
\PR{\Omega_2(\eps)^c} \le \PR{\left|s_a(\eps) - s \cdot\left(\frac{\vol{\Pos(\eps) \setminus \Pos}}{\vol{\RP^N}}\right)\right| > s^{\nicefrac{1}{2} + \beta}\cdot\left(\frac{\vol{\Pos(\eps) \setminus \Pos}}{\vol{\RP^N}}\right)} < \frac{\delta}{3},
\ee
and, if $s > \left(\frac{3\log{\nicefrac{6}{\delta}}}{\frac{\vol{\Pos}}{\vol{\RP^N}}}\right)^{\nicefrac{1}{2\beta}}$,
\be
\label{eqn:sp-tail}
\PR{\Omega_3^c} \le \PR{\left|s_p - s \cdot \left(\frac{\vol{\Pos}}{\vol{\RP^N}}\right)\right| > s^{\nicefrac{1}{2} + \beta}\cdot\left(\frac{\vol{\Pos}}{\vol{\RP^N}}\right)} < \frac{\delta}{3}.
\ee
Setting $\tilde{s}_1 = \mathrm{max}\left\{\left(\frac{3\log{\nicefrac{6}{\delta}}}{1 - \frac{\vol{\Pos(\eps)}}{\vol{\RP^N}}}\right)^{\nicefrac{1}{2\beta}}, \left(\frac{3\log{\nicefrac{6}{\delta}}}{\frac{\vol{\Pos(\eps) \setminus \Pos}}{\vol{\RP^N}}}\right)^{\nicefrac{1}{2\beta}}, \left(\frac{3\log{\nicefrac{6}{\delta}}}{\frac{\vol{\Pos}}{\vol{\RP^N}}}\right)^{\nicefrac{1}{2\beta}}\right\}$, observe that if $s \ge \tilde{s}_1$,
\begin{align}
\PR{\Omega_1(\eps) \cap \Omega_2(\eps) \cap \Omega_3} &= 1 - \PR{\Omega_1(\eps)^c \cup \Omega_2(\eps)^c \cup \Omega_3^c} \\
&\ge 1 - \PR{\Omega_1(\eps)^c} + \PR{\Omega_2(\eps)^c} + \PR{\Omega_3^c} \eqcomment{by a union bound}\\
&\ge 1 - 3\cdot\frac{\delta}{3} = 1 - \delta \eqcomment{by \eqref{eqn:se-tail}, \eqref{eqn:sa-tail}, and \eqref{eqn:sp-tail}}.
\end{align}
Finally, from \eqref{eqn:concentration-limit}, we can easily deduce that
\be
\lim_{s \to \infty} \PR{\Omega_1(\eps) \cap \Omega_2(\eps) \cap \Omega_3} = 1.
\ee
This completes the proof of Proposition \ref{prop:good-event-probability}.
\end{proof}

\end{document}